\documentclass[12pt]{amsart}
\usepackage{amsmath}
\usepackage{amsxtra}
\usepackage{amscd}
\usepackage{amsthm}
\usepackage{amsfonts}
\usepackage{amssymb}
\usepackage{eucal}
\usepackage{verbatim}
\usepackage[all]{xy}
\usepackage{color}
\definecolor{deepjunglegreen}{rgb}{0.0, 0.29, 0.29}
\definecolor{darkspringgreen}{rgb}{0.09, 0.45, 0.27}

\usepackage{stmaryrd}

\makeatletter
\usepackage{etex,etoolbox,needspace}
\pretocmd\section{\Needspace*{4\baselineskip}}{}{}
\makeatletter

\usepackage[pdftex,bookmarks=false,colorlinks=true,citecolor=darkspringgreen,debug=true,
  naturalnames=true,pdfnewwindow=true]{hyperref}

\newtheorem{thm}{Theorem}[subsection]

\newtheorem{cor}[thm]{Corollary}

\newtheorem{lem}[thm]{Lemma}
\newtheorem{prop}[thm]{Proposition}
\newtheorem{conj}[thm]{Conjecture}

\theoremstyle{definition}
\newtheorem{defn}[thm]{Definition}

\theoremstyle{remark}
\newtheorem{rem}[thm]{Remark}

\newcommand{\nc}{\newcommand}
\nc{\renc}{\renewcommand} \nc{\ssec}{\subsection}
\nc{\sssec}{\subsubsection}

\nc{\on}{\operatorname} \nc{\wh}{\widehat}
\nc\ol{\overline} \nc\ul{\underline} \nc\wt{\widetilde}

\newcommand{\red}[1]{{\color{red}#1}}

\nc{\BA}{{\mathbb{A}}} \nc{\BC}{{\mathbb{C}}} \nc{\BF}{{\mathbb{F}}}
\nc{\BD}{{\mathbb{D}}} \nc{\BG}{{\mathbb{G}}} \nc{\BQ}{{\mathbb{Q}}}
\nc{\BM}{{\mathbb{M}}} \nc{\BN}{{\mathbb{N}}} \nc{\BO}{{\mathbb{O}}}
\nc{\BP}{{\mathbb{P}}} \nc{\BR}{{\mathbb{R}}}
\nc{\BZ}{{\mathbb{Z}}} \nc{\BS}{{\mathbb{S}}} \nc{\BW}{{\mathbb{W}}}

\nc{\CA}{{\mathcal{A}}} \nc{\CB}{{\mathcal{B}}} \nc{\CalC}{{\mathcal{C}}} \nc{\CalD}{{\mathcal{D}}}
\nc{\CE}{{\mathcal{E}}} \nc{\CF}{{\mathcal{F}}}
\nc{\CG}{{\mathcal{G}}} \nc{\CH}{{\mathcal{H}}}
\nc{\CI}{{\mathcal{I}}} \nc{\CK}{{\mathcal{K}}} \nc{\CL}{{\mathcal{L}}}
\nc{\CM}{{\mathcal{M}}} \nc{\CN}{{\mathcal{N}}}
\nc{\CO}{{\mathcal{O}}} \nc{\CP}{{\mathcal{P}}}
\nc{\CQ}{{\mathcal{Q}}} \nc{\CR}{{\mathcal{R}}}
\nc{\CS}{{\mathcal{S}}} \nc{\CT}{{\mathcal{T}}}
\nc{\CU}{{\mathcal{U}}} \nc{\CV}{{\mathcal{V}}}  \nc{\CY}{{\mathcal Y}}
\nc{\CW}{{\mathcal{W}}} \nc{\CZ}{{\mathcal{Z}}}

\nc{\cM}{{\check{\mathcal M}}{}} \nc{\csM}{{\check{\mathcal A}}{}}
\nc{\oM}{{\overset{\circ}{\mathcal M}}{}}
\nc{\obM}{{\overset{\circ}{\mathbf M}}{}}
\nc{\oCA}{{\overset{\circ}{\mathcal A}}{}}
\nc{\obA}{{\overset{\circ}{\mathbf A}}{}}
\nc{\ooM}{{\overset{\circ}{M}}{}}
\nc{\osM}{{\overset{\circ}{\mathsf M}}{}}
\nc{\vM}{{\overset{\bullet}{\mathcal M}}{}}
\nc{\nM}{{\underset{\bullet}{\mathcal M}}{}}
\nc{\oD}{{\overset{\circ}{\mathcal D}}{}}
\nc{\obD}{{\overset{\circ}{\mathbf D}}{}}
\nc{\oA}{{\overset{\circ}{\mathbb A}}{}}
\nc{\op}{{\overset{\bullet}{\mathbf p}}{}}
\nc{\cp}{{\overset{\circ}{\mathbf p}}{}}
\nc{\oU}{{\overset{\bullet}{\mathcal U}}{}}
\nc{\ofZ}{{\overset{\circ}{\mathfrak Z}}{}}

\nc{\ff}{{\mathfrak{f}}} \nc{\fv}{{\mathfrak{v}}}
\nc{\fa}{{\mathfrak{a}}} \nc{\fb}{{\mathfrak{b}}}
\nc{\fd}{{\mathfrak{d}}} \nc{\fe}{{\mathfrak{e}}}
\nc{\fg}{{\mathfrak{g}}} \nc{\fgl}{{\mathfrak{gl}}}
\nc{\fh}{{\mathfrak{h}}} \nc{\fri}{{\mathfrak{i}}}
\nc{\fj}{{\mathfrak{j}}} \nc{\fk}{{\mathfrak{k}}} \nc{\fl}{{\mathfrak{l}}}
\nc{\fm}{{\mathfrak{m}}} \nc{\fn}{{\mathfrak{n}}}
\nc{\ft}{{\mathfrak{t}}} \nc{\fu}{{\mathfrak{u}}}
\nc{\fw}{{\mathfrak{w}}} \nc{\fz}{{\mathfrak{z}}}
\nc{\fp}{{\mathfrak{p}}} \nc{\fq}{{\mathfrak{q}}} \nc{\frr}{{\mathfrak{r}}}
\nc{\fs}{{\mathfrak{s}}} \nc{\fsl}{{\mathfrak{sl}}}
\nc{\fso}{{\mathfrak{so}}} \nc{\fsp}{{\mathfrak{sp}}} \nc{\osp}{{\mathfrak{osp}}}
\nc{\hsl}{{\widehat{\mathfrak{sl}}}}
\nc{\hgl}{{\widehat{\mathfrak{gl}}}}
\nc{\hg}{{\widehat{\mathfrak{g}}}}
\nc{\chg}{{\widehat{\mathfrak{g}}}{}^\vee}
\nc{\hn}{{\widehat{\mathfrak{n}}}}
\nc{\chn}{{\widehat{\mathfrak{n}}}{}^\vee}

\nc{\fA}{{\mathfrak{A}}} \nc{\fB}{{\mathfrak{B}}} \nc{\fC}{{\mathfrak{C}}}
\nc{\fD}{{\mathfrak{D}}} \nc{\fE}{{\mathfrak{E}}}
\nc{\fF}{{\mathfrak{F}}} \nc{\fG}{{\mathfrak{G}}} \nc{\fH}{{\mathfrak{H}}}
\nc{\fI}{{\mathfrak{I}}} \nc{\fJ}{{\mathfrak{J}}}
\nc{\fK}{{\mathfrak{K}}} \nc{\fL}{{\mathfrak{L}}}
\nc{\fM}{{\mathfrak{M}}} \nc{\fN}{{\mathfrak{N}}}
\nc{\frP}{{\mathfrak{P}}} \nc{\fQ}{{\mathfrak{Q}}}
\nc{\fS}{{\mathfrak{S}}} \nc{\fT}{{\mathfrak{T}}} \nc{\fU}{{\mathfrak{U}}}
\nc{\fV}{{\mathfrak{V}}} \nc{\fW}{{\mathfrak{W}}}
\nc{\fX}{{\mathfrak{X}}} \nc{\fY}{{\mathfrak{Y}}}
\nc{\fZ}{{\mathfrak{Z}}}

\nc{\ba}{{\mathbf{a}}}
\nc{\bb}{{\mathbf{b}}} \nc{\bc}{{\mathbf{c}}} \nc{\be}{{\mathbf{e}}}
\nc{\bg}{{\mathbf{g}}} \nc{\bj}{{\mathbf{j}}} \nc{\bm}{{\mathbf{m}}}
\nc{\bn}{{\mathbf{n}}} \nc{\bp}{{\mathbf{p}}}
\nc{\bq}{{\mathbf{q}}} \nc{\br}{{\mathbf{r}}} \nc{\bt}{{\mathbf{t}}}
\nc{\bfu}{{\mathbf{u}}} \nc{\bv}{{\mathbf{v}}}
\nc{\bx}{{\mathbf{x}}} \nc{\by}{{\mathbf{y}}} \nc{\bz}{{\mathbf{z}}}
\nc{\bw}{{\mathbf{w}}} \nc{\bA}{{\mathbf{A}}}
\nc{\bB}{{\mathbf{B}}} \nc{\bC}{{\mathbf{C}}}
\nc{\bD}{{\mathbf{D}}} \nc{\bF}{{\mathbf{F}}} \nc{\bG}{{\mathbf{G}}}
\nc{\bH}{{\mathbf{H}}} \nc{\bI}{{\mathbf{I}}} \nc{\bJ}{{\mathbf{J}}}
\nc{\bK}{{\mathbf{K}}} \nc{\bM}{{\mathbf{M}}} \nc{\bN}{{\mathbf{N}}}
\nc{\bO}{{\mathbf{O}}} \nc{\bS}{{\mathbf{S}}} \nc{\bT}{{\mathbf{T}}}
\nc{\bU}{{\mathbf{U}}} \nc{\bV}{{\mathbf{V}}} \nc{\bW}{{\mathbf{W}}}
\nc{\bX}{{\mathbf{X}}}
\nc{\bY}{{\mathbf{Y}}} \nc{\bP}{{\mathbf{P}}}
\nc{\bZ}{{\mathbf{Z}}} \nc{\bh}{{\mathbf{h}}}

\nc{\sA}{{\mathsf{A}}} \nc{\sB}{{\mathsf{B}}}
\nc{\sC}{{\mathsf{C}}} \nc{\sD}{{\mathsf{D}}}
\nc{\sE}{{\mathsf{E}}} \nc{\sF}{{\mathsf{F}}} \nc{\sG}{{\mathsf{G}}}
\nc{\sI}{{\mathsf{I}}} \nc{\sK}{{\mathsf{K}}} \nc{\sL}{{\mathsf{L}}}
\nc{\sfm}{{\mathsf{m}}} \nc{\sM}{{\mathsf{M}}} \nc{\sN}{{\mathsf{N}}}
\nc{\sO}{{\mathsf{O}}} \nc{\sQ}{{\mathsf{Q}}} \nc{\sP}{{\mathsf{P}}}
\nc{\sT}{{\mathsf{T}}} \nc{\sZ}{{\mathsf{Z}}}
\nc{\sV}{{\mathsf{V}}} \nc{\sW}{{\mathsf{W}}}
\nc{\sfp}{{\mathsf{p}}} \nc{\sq}{{\mathsf{q}}} \nc{\sr}{{\mathsf{r}}}
\nc{\sfs}{{\mathsf{s}}} \nc{\st}{{\mathsf{t}}} \nc{\sfb}{{\mathsf{b}}}
\nc{\sfc}{{\mathsf{c}}} \nc{\sd}{{\mathsf{d}}}
\nc{\sz}{{\mathsf{z}}}

\nc{\tA}{{\widetilde{\mathbf{A}}}}
\nc{\tB}{{\widetilde{\mathcal{B}}}}
\nc{\tg}{{\widetilde{\mathfrak{g}}}} \nc{\tG}{{\widetilde{G}}}
\nc{\TM}{{\widetilde{\mathbb{M}}}{}}
\nc{\tO}{{\widetilde{\mathsf{O}}}{}}
\nc{\tU}{{\widetilde{\mathfrak{U}}}{}} \nc{\TZ}{{\tilde{Z}}}
\nc{\tx}{{\tilde{x}}} \nc{\tbv}{{\tilde{\bv}}}
\nc{\tfP}{{\widetilde{\mathfrak{P}}}{}} \nc{\tz}{{\tilde{\zeta}}}
\nc{\tmu}{{\tilde{\mu}}}

\nc{\urho}{\underline{\rho}} \nc{\uB}{\underline{B}}
\nc{\uC}{{\underline{\mathbb{C}}}} \nc{\ui}{\underline{i}}
\nc{\uj}{\underline{j}} \nc{\ofP}{{\overline{\mathfrak{P}}}}
\nc{\oB}{{\overline{\mathcal{B}}}}
\nc{\og}{{\overline{\mathfrak{g}}}} \nc{\oI}{{\overline{I}}}

\nc{\eps}{\varepsilon} \nc{\hrho}{{\hat{\rho}}}
\nc{\blambda}{{\boldsymbol{\lambda}}} \nc{\bmu}{{\boldsymbol{\mu}}} \nc{\bnu}{{\boldsymbol{\nu}}}
\nc{\btheta}{{\boldsymbol{\theta}}} \nc{\bzeta}{{\boldsymbol{\zeta}}} \nc{\bta}{{\boldsymbol{\eta}}}

\nc{\one}{{\mathbf{1}}} \nc{\two}{{\mathbf{t}}}

\nc{\Sym}{\mathop{\operatorname{\rm Sym}}}
\nc{\Tot}{{\mathop{\operatorname{\rm Tot}}}}
\nc{\Spec}{\mathop{\operatorname{\rm Spec}}}
\nc{\Ker}{{\mathop{\operatorname{\rm Ker}}}}
\nc{\Isom}{{\mathop{\operatorname{\rm Isom}}}}
\nc{\Hilb}{{\mathop{\operatorname{\rm Hilb}}}}
\nc{\deeq}{{\mathop{\operatorname{\rm deeq}}}}
\nc{\End}{{\mathop{\operatorname{\rm End}}}}
\nc{\Ext}{{\mathop{\operatorname{\rm Ext}}}}
\nc{\Hom}{{\mathop{\operatorname{\rm Hom}}}}
\nc{\CHom}{{\mathop{\operatorname{{\mathcal{H}}\it om}}}}
\nc{\GL}{{\mathop{\operatorname{\rm GL}}}}
\nc{\SL}{{\mathop{\operatorname{\rm SL}}}}
\nc{\SO}{{\mathop{\operatorname{\rm SO}}}}
\nc{\Sp}{{\mathop{\operatorname{\rm Sp}}}}
\nc{\OSp}{{\mathop{\operatorname{\rm SOSp}}}}
\nc{\gr}{{\mathop{\operatorname{\rm gr}}}}
\nc{\Id}{{\mathop{\operatorname{\rm Id}}}}
\nc{\perf}{{\mathop{\operatorname{\rm perf}}}}
\nc{\defi}{{\mathop{\operatorname{\rm def}}}}
\nc{\length}{{\mathop{\operatorname{\rm length}}}}
\nc{\supp}{{\mathop{\operatorname{\rm supp}}}}
\nc{\HC}{{\mathcal H}{\mathcal C}}

\nc{\pr}{{\operatorname{pr}}}
\nc{\Cliff}{{\mathsf{Cliff}}}
\nc{\loc}{{\operatorname{loc}}}
\nc{\Fl}{{\mathbf{Fl}}} \nc{\Ffl}{{\mathcal{F}\ell}}
\nc{\Fib}{{\mathsf{Fib}}}
\nc{\Coh}{{\mathsf{Coh}}} \nc{\FCoh}{{\mathsf{FCoh}}}
\nc{\Perf}{{\mathsf{Perf}}}
\nc{\wtimes}{\mathbin{\widetilde\times}}

\nc{\reg}{{\text{\rm reg}}}
\nc{\self}{{\text{\rm self}}}
\nc{\gvee}{{\mathfrak g}^{\!\scriptscriptstyle\vee}}
\nc{\tvee}{{\mathfrak t}^{\!\scriptscriptstyle\vee}}
\nc{\nvee}{{\mathfrak n}^{\!\scriptscriptstyle\vee}}
\nc{\bvee}{{\mathfrak b}^{\!\scriptscriptstyle\vee}}
       \nc{\rhovee}{\rho^{\!\scriptscriptstyle\vee}}

\nc{\cplus}{{\mathbf{C}_+}} \nc{\cminus}{{\mathbf{C}_-}}
\nc{\cthree}{{\mathbf{C}_*}} \nc{\Qbar}{{\bar{Q}}}

\newcommand\iso{\mathbin{\vphantom{j^{X^2}}\smash{\overset{\sim}{\vphantom{\rule{0pt}{0.20em}}\smash{\longrightarrow}}}}}
\nc{\Gtimes}{\vphantom{j^{X^2}}\smash{\overset{G}{\vphantom{\rule{0pt}{0.30em}}\smash{\times}}}}
\nc{\sGtimes}{\vphantom{j^{X^2}}\smash{\overset{\mathsf G}{\vphantom{\rule{0pt}{0.30em}}\smash{\times}}}}

\nc{\bOmega}{{\overline{\Omega}}}

\nc{\seq}[1]{\stackrel{#1}{\sim}}

\nc{\aff}{{\operatorname{aff}}}
\nc{\fin}{{\operatorname{fin}}}
\nc{\mir}{{\operatorname{mir}}}
\nc{\triv}{{\operatorname{triv}}}
\nc{\ext}{{\operatorname{ext}}}
\nc{\righ}{{\operatorname{right}}}
\nc{\lef}{{\operatorname{left}}}
\nc{\forg}{{\operatorname{forg}}}
\nc{\fid}{{\operatorname{fd}}}
\nc{\odd}{{\operatorname{odd}}}
\nc{\even}{{\operatorname{even}}}
\nc{\modu}{{\operatorname{-mod}}}
\nc{\Gr}{{\mathbf{Gr}}}
\nc{\FT}{{\operatorname{FT}}}
\nc{\Mat}{{\operatorname{Mat}}}
\nc{\MSt}{{\operatorname{MSt}}}
\nc{\sph}{{\operatorname{sph}}}
\nc{\GR}{{\mathbf{Gr}}}
\nc{\Perv}{{\operatorname{Perv}}}
\nc{\Rep}{{\operatorname{Rep}}}
\nc{\Ind}{{\operatorname{Ind}}}
\nc{\IC}{{\operatorname{IC}}}
\nc{\Bun}{{\operatorname{Bun}}}
\nc{\Proj}{{\operatorname{Proj}}}
\nc{\Stab}{{\operatorname{Stab}}}
\nc{\pt}{{\operatorname{pt}}}
\nc{\bfmu}{{\boldsymbol{\mu}}}
\nc{\bfomega}{{\boldsymbol{\omega}}}
\nc{\calM}{\mathcal M}
\nc{\calA}{\mathcal A}
\nc{\calO}{\mathcal O}
\nc{\CC}{\mathbb C}
\nc{\calN}{\mathcal N}
\nc{\grg}{\mathfrak g}
\nc{\dslash}{/\!\!/}
\nc{\tslash}{/\!\!/\!\!/}
\nc\grt{\mathfrak t}
\nc\bfM{\mathbf M}
\nc\bfN{\mathbf N}
\nc\Sig{\Sigma}
\nc\ZZ{\mathbb{Z}}
\nc\calC{\mathcal C}
\nc\calF{\mathcal F}
\nc\calX{\mathcal X}
\nc\calY{\mathcal Y}
\nc\QCoh{\operatorname{QCoh}}
\nc\IndCoh{\operatorname{IndCoh}}
\nc\Maps{\operatorname{Maps}}
\nc\Dmod{D-\operatorname{mod}}
\newcommand\Hecke{\operatorname{Hecke}}
\nc{\calD}{\mathcal D}
\nc\bfO{\mathbf O}

\nc\GG{\mathbb G}
\nc\calK{\mathcal K}
\nc{\calG}{\mathcal G}
\nc\RHom{\operatorname{RHom}}
\nc\Res{\operatorname{Res}}
\nc\Av{\operatorname{Av}}
\nc\grs{\mathfrak s}
\nc{\tilX}{\widetilde X}
\nc\calB{\mathcal B}
\nc\calS{\mathcal S}
\nc\calT{\mathcal T}
\nc\calZ{\mathcal Z}
\nc\LS{\operatorname{LocSys}}
\nc\bfL{\on{\mathbf L}}

\newcommand*\circled[1]
{*{#1}}
\makeatletter
\newcommand{\raisemath}[1]{\mathpalette{\raisem@th{#1}}}
\newcommand{\raisem@th}[3]{\raisebox{#1}{$#2#3$}}
\nc{\binlim}[2][]{\def\@tempa{#1}\@ifnextchar^{\@binlim{#2}}{\@binlim{#2}^{}}}
\def\@binlim#1^#2{\mathbin{\@ifempty{#2}{\mathop{#1}}{\mathop{#1}\@xp\displaylimits\@tempa^{#2}}}}

\makeatother

\nc\cX{{\mathcal X}}

\newcommand{\dbkts}[1]{[\![#1]\!]}
\newcommand{\dprts}[1]{(\!(#1)\!)}
\nc\Gm{{\mathbb G_m}}
\nc{\isoto}
    {\stackrel{\displaystyle\raisemath{-.2ex}{\sim}}\to}
\nc\cD\CalD
\nc\setm\setminus
\nc\cE\CE
\renc\Hecke{\mathit{\CH\kern-.2ex ecke}}
\nc\bsl\backslash
\nc\Fq{\mathbb F_q}
\nc\x\times
\nc\bGO{{\bG_\bO}}
\nc\bla\blambda
\nc\blt\bullet
\nc\opp{{\on{op}}}
\nc\srel\stackrel
\nc\oast\circledast
\nc\tbx{\binlim{\widetilde\boxtimes{}}}
\nc\into\hookrightarrow
\nc\otto\leftrightarrow
\renc\setminus\smallsetminus
\nc\phitau{\varphi\tau}
\newenvironment{i-ii-iii}{%
\begin{enumerate}
}%
{\end{enumerate}}
\nc\ceil[1]{\lceil#1\rceil}  \nc\floor[1]{\lfloor#1\rfloor}
\nc\Lie{\on{Lie}}
\nc\sS{{\mathsf S}}
\nc\vvv{\ensuremath{\red\surd}}
\nc\la\lambda
 \let\arXiv\arxiv
\nc\kap{\kappa}
\nc\gra{\mathfrak a}
\nc\gl{\mathfrak{gl}}
\nc\sTr{\operatorname{sTr}}
\nc\hatG{\widehat{G}}
\nc\calL{\mathcal L}
\nc\Whit{\operatorname{Whit}}
\nc\KL{\operatorname{KL}}

\newcommand\Killing{\operatorname{Killing}}

\makeatletter
\renewcommand{\subsection}{\@startsection{subsection}{2}{0pt}{-3ex
plus -1ex minus -0.2ex}{-2mm plus -0pt minus
-2pt}{\normalfont\bfseries}} \makeatother

\numberwithin{equation}{subsection}
\allowdisplaybreaks
\nc\mto{\mapsto }
\nc\en{\enspace }

\begin{document}

\author[A.Braverman]{Alexander Braverman}
\address{Department of Mathematics, University of Toronto and Perimeter Institute
of Theoretical Physics, Waterloo, Ontario, Canada, N2L 2Y5;
\newline Skolkovo Institute of Science and Technology}
\email{braval@math.toronto.edu}

\author[M.Finkelberg]{Michael Finkelberg}
\address{National Research University Higher School of Economics, Russian Federation,
  Department of Mathematics, 6 Usacheva st, 119048 Moscow;
\newline Skolkovo Institute of Science and Technology;
\newline Institute for the Information Transmission Problems}
\email{fnklberg@gmail.com}


\author[R.Travkin]{Roman Travkin}
\address{Skolkovo Institute of Science and Technology, Moscow, Russia}
\email{roman.travkin2012@gmail.com}

\title
{Orthosymplectic Satake equivalence}
\dedicatory{To the memory of Elena V.~Glivenko}




\begin{abstract}
  This is a companion paper of~\cite{bfgt}.
  We prove an equivalence relating representations of a degenerate orthosymplectic supergroup
  with the category of $\SO(N-1,\BC[\![t]\!])$-equivariant perverse sheaves on the affine
  Grassmannian of $\SO_N$. We explain how this equivalence fits into a more general framework
  of conjectures due to Gaiotto and to Ben-Zvi, Sakellaridis and Venkatesh.
\end{abstract}

\maketitle

\tableofcontents

\section{Introduction}
\label{intro}

\subsection{Reminder on~\cite{bfgt}}
\label{recall}
Recall one of the results of~\cite{bfgt}. We consider the Lie superalgebra $\fgl(N-1|N)$
of endomorphisms of a super vector space
$\BC^{N-1|N}$, and the corresponding algebraic supergroup $\GL(N-1|N)=\on{Aut}(\BC^{N-1|N})$.
We also consider a degenerate version $\ul\fgl(N-1|N)$ where the
supercommutator of the even elements (with even or odd elements) is the same as in $\fgl(N-1|N)$,
while the supercommutator of any two odd elements is set to be zero.
In other words, the even part $\ul\fgl(N-1|N){}_{\bar0}=\fgl_{N-1}\oplus\fgl_N$ acts naturally on the
odd part $\ul\fgl(N-1|N){}_{\bar1}=\Hom(\BC^{N-1},\BC^N)\oplus\Hom(\BC^N,\BC^{N-1})$, but the
supercommutator $\ul\fgl(N-1|N){}_{\bar1}\times\ul\fgl(N-1|N){}_{\bar1}\to\ul\fgl(N-1|N){}_{\bar0}$
equals zero.

The category of
finite dimensional representations of the corresponding supergroup $\ul\GL(N-1|N)$ (in vector
superspaces) is denoted $\Rep(\ul\GL(N-1|N))$, and its bounded derived category is denoted
$D^b\Rep(\ul\GL(N-1|N))$. In~\cite{bfgt} we construct an equivalence $\Psi$ from
$D^b\Rep(\ul\GL(N-1|N))$ to
the bounded equivariant derived constructible category $SD^b_{\GL(N-1,\bO)}(\Gr_{\GL_N})$ with
coefficients in vector superspaces.
Here $\bO=\BC\dbkts{t}\subset\BC\dprts{t}=\bF$, and $\Gr_{\GL_N}=\GL(N,\bF)/\GL(N,\bO)$.
This equivalence enjoys the following favorable properties, reminiscent of the classical
geometric Satake equivalence (e.g.\ $\Rep(\GL_N)\iso\Perv_{\GL(N,\bO)}(\Gr_{\GL_N})$):

(i) $\Psi$ is exact with respect to the tautological $t$-structure on $D^b\Rep(\ul\GL(N-1|N))$ with
the heart $\Rep(\ul\GL(N-1|N))$ and the perverse $t$-structure on $SD^b_{\GL(N-1,\bO)}(\Gr_{\GL_N})$
with the heart $S\Perv_{\GL(N-1,\bO)}(\Gr_{\GL_N})$.

(ii) $\Psi$ takes the tensor product of $\ul\GL(N-1|N)$-modules to the fusion product~$\star$
on $SD^b_{\GL(N-1,\bO)}(\Gr_{\GL_N})$.

As a corollary, we derive an equivalence $SD^b_{\GL(N-1,\bO)}(\Gr_{\GL_N})\simeq
D^b\big(S\Perv_{\GL(N-1,\bO)}(\Gr_{\GL_N})\big)$ in sharp contrast with the classical geometric
Satake category, where e.g.\ $\Perv_{\GL(N,\bO)}(\Gr_{\GL_N})$ is semisimple, and its derived
category $D^b\big(\Perv_{\GL(N,\bO)}(\Gr_{\GL_N})\big)$ is not equivalent to
$D^b_{\GL(N,\bO)}(\Gr_{\GL_N})$.

The equivalence $\Psi$ was obtained in~\cite{bfgt} as a byproduct of a construction of a
similar equivalence for the mirabolic affine Grassmannian. In case $N=2$, the equivalence
$\Psi$ was constructed earlier in~\cite{br} in a much more direct way.

\subsection{Orthosymplectic Satake equivalence}
\label{OSE}
One of the goals of the present paper is to generalize the direct approach of~\cite{br} to
the study of $SD^b_{\SO(N-1,\bO)}(\Gr_{\SO_N})$
(note that $\SO_2\simeq\GL_1$, and $\SO_3\simeq\on{PGL}_2$).\footnote{In fact, this generalization
  works similarly for the original problem: for the general linear group $\GL$ in place of the
  special orthogonal group $\SO$.}
The corresponding supergroup turns out to be a degeneration $\ul\sG$ of an orthosymplectic algebraic
supergroup $\sG$ whose even part $\sG_{\bar0}$ is the Langlands dual of $\SO_{N-1}\times\SO_N$.
In order to describe it explicitly we will distinguish two cases, depending on parity of $N$.
Throughout the paper we assume $N\geq3$.

(a) {\em odd:} If $N=2n+1$, we set
$V_0=\BC^{2n}$ equipped with a nondegenerate symmetric bilinear form $(,)$, and
$V_1=\BC^{2n}$ equipped with a nondegenerate skew-symmetric bilinear form $\langle,\rangle$.

(b) {\em even:} If $N=2n$, we set
$V_0=\BC^{2n}$ equipped with a nondegenerate symmetric bilinear form $(,)$, and
$V_1=\BC^{2n-2}$ equipped with a nondegenerate skew-symmetric bilinear form $\langle,\rangle$.

We consider the Lie superalgebra $\fgl(V_0|V_1)$ of endomorphisms of a super vector space
$V_0\oplus\Pi V_1$, and the corresponding algebraic supergroup $\GL(V_0|V_1)$. The super vector space
$V_0\oplus\Pi V_1$ is equipped with the bilinear form $(,)\oplus\langle,\rangle$, and the
orthosymplectic Lie superalgebra $\bg:=\osp(V_0|V_1)\subset\fgl(V_0|V_1)$ is formed by all the
endomorphisms preserving the above bilinear form
(in  the Lie superalgebra sense). The corresponding algebraic  supergroup
$\sG:=\OSp(V_0|V_1)\subset\GL(V_0|V_1)$, by definition, has the even part
$\sG_{\bar0}=\SO(V_0)\times\Sp(V_1)$. Accordingly, the even part $\bg_{\bar0}=\fso(V_0)\oplus\fsp(V_1)$
acts naturally on the odd part $\bg_{\bar1}=V_0\otimes\Pi V_1$.

We also consider a degenerate version $\ul\bg=\ul\osp(V_0|V_1)$ where the supercommutator of the
even elements (with even or odd elements) is the same as in $\osp(V_0|V_1)$, while the
supercommutator of any two odd elements is set to be zero. The corresponding Lie supergroup
is denoted $\ul\sG=\ul\OSp(V_0|V_1)$; its even part is equal to
$\ul\sG{}_{\bar0}=\sG_{\bar0}=\SO(V_0)\times\Sp(V_1)$.

The category of finite dimensional representations of $\ul\sG$ (in super vector spaces) is
denoted $\Rep(\ul\sG)$, and its bounded derived category is denoted $D^b\Rep(\ul\sG)$.

In our main Theorem~\ref{main} we construct an equivalence $\Xi$ from $D^b\Rep(\ul\OSp(V_0|V_1))$ to
the bounded equivariant derived constructible category $SD^b_{\SO(N-1,\bO)}(\Gr_{\SO_N})$ with
coefficients in vector superspaces. This equivalence enjoys the favorable properties similar
to the properties of the equivalence $\Psi$ of~\S\ref{recall}:

(i) $\Xi$ is exact with respect to the tautological $t$-structure on $D^b\Rep(\ul\OSp(V_0|V_1))$
with the heart $\Rep(\ul\OSp(V_0|V_1))$ and the perverse $t$-structure on
$SD^b_{\SO(N-1,\bO)}(\Gr_{\SO_N})$ with the heart $S\Perv_{\SO(N-1,\bO)}(\Gr_{\SO_N})$.

(ii) $\Xi$ takes the tensor product of $\ul\OSp(V_0|V_1)$-modules to the fusion product~$\star$
on $SD^b_{\SO(N-1,\bO)}(\Gr_{\SO_N})$.

\begin{rem}
  One of the key ingredients in the proof of Theorem~\ref{main} is Ginzburg's theorem~\cite{g2}
  identifying the (equivariant) Exts between IC-sheaves on a variety $X$ with the
  homomorphisms over the (equivariant) cohomology ring of $X$ between the (equivariant)
  cohomology of $X$ with coefficients in the above IC-sheaves. One of the necessary conditions
  for Ginzburg's theorem is the existence of a cellular decomposition of $X$ such that the
  IC-sheaves in question are smooth along cells. A standard application of Ginzburg's theorem
  is to $\SO(N,\bO)$-equivariant IC-sheaves on $\Gr_{\SO_N}$. But in our situation there
  is {\em no} cellular decomposition of $\Gr_{\SO_N}$ such that all the $\SO(N-1,\bO)$-equivariant
  IC-sheaves are smooth along cells. However, our proof of~Theorem~\ref{main}
  establishes along the way Ginzburg's theorem {\em a posteriori}.
\end{rem}

\subsection{Conjectures of Ben-Zvi, Sakellaridis and Venkatesh}
By definition of the degenerate orthosymplectic algebra $\ul\bg=\ul\osp(V_0|V_1)$,
its odd part $\bg_{\bar1}$ is a Lie superalgebra with trivial supercommutator, so that its
universal enveloping algebra is a (finite-dimensional) exterior algebra $\Lambda$.
The derived category $D\Rep(\ul\sG)$ is nothing but the derived category
$SD_\fid^{\sG_{\bar0}}(\Lambda)$ of finite dimensional $\sG_{\bar0}$-equivariant super dg-modules
over $\Lambda$ (viewed as a dg-algebra with trivial differential).
There is a Koszul equivalence
$SD_\perf^{\sG_{\bar0}}(\fG^\bullet)\iso SD_\fid^{\sG_{\bar0}}(\Lambda)\cong D^b\Rep(\ul\sG)$
where $\fG^\bullet=\Sym(\bg_{\bar1}[-1])$
(we use  the trace paring to identify $\bg_{\bar1}$ with $\bg_{\bar1}^*$)
 is a dg-algebra with trivial differential,
and $SD_\perf^{\sG_{\bar0}}(\fG^\bullet)$ stands for the derived category of $\sG_{\bar0}$-equivariant
perfect dg-modules over $\fG^\bullet$. Precomposing the equivalence
$\Xi\colon D^b\Rep(\ul\OSp(V_0|V_1))\iso SD^b_{\SO(N-1,\bO)}(\Gr_{\SO_N})$ with the Koszul equivalence,
we obtain an equivalence
$\Phi\colon SD_\perf^{\sG_{\bar0}}(\fG^\bullet)\iso SD^b_{\SO(N-1,\bO)}(\Gr_{\SO_N})$.

One advantage of $\Phi$ (over $\Xi$) is that it admits a straightforward quantization $\Phi_\hbar$
describing the category $SD^b_{\SO(N-1,\bO)\rtimes\BC^\times}(\Gr_{\SO_N})$ with equivariance extended by the
loop rotations, see~Theorem~\ref{quantum}.

Another advantage is that the subcategory of
$SD^b_{\SO(N-1,\bO)\rtimes\BC^\times}(\Gr_{\SO_N})$ formed by all the objects that are compact as the
objects of {\em unbounded} category $SD_{\SO(N-1,\bO)\rtimes\BC^\times}(\Gr_{\SO_N})$ is obtained by
applying $\Phi$ to the subcategory of $SD_\perf^{\sG_{\bar0}}(\fG^\bullet)$ formed by all the
objects with the nilpotent support condition, see~Theorem~\ref{main}.

In yet another direction, as explained in~\cite[\S1.7]{bfgt}, this equivalence is an instance of
the Periods---$L$-functions duality conjectures of D.~Ben-Zvi, Y.~Sakellaridis and A.~Venkatesh.
Their conjectures predict, among other things, that given a reductive group $\on{G}$ and its
spherical homogeneous variety $X=\on{G}/\on{H}$, there is a subgroup
$\on{G}^\vee_X\subset\on{G}^\vee$, its
graded representation $V^\vee_X=\bigoplus_{i\in\BZ}V^\vee_{X,i}[i]$, and an equivalence
$D\Coh(V^\vee_X/\on{G}^\vee_X)=D\Coh\big((\bigoplus_{i\in\BZ}V^\vee_{X,i}[i])/\on{G}^\vee_X\big)
\simeq D_{\on{G}(\bO)}(X(\bF))$.
For a partial list of examples, see the table at the end of~\cite{sa}. The relevant
representations $V^\vee_X$ (constructed in terms of the Luna diagram of $X$) can be
read off from the 4-th column of the table.

It turns out that the case of Example~14 of~\cite{sa} is the above equivalence $\Phi$,
or rather its version with coefficients in usual vector spaces (as opposed to super vector
spaces) $D_\perf^{\sG_{\bar0}}(\fG^\bullet)\iso D^b_{\SO(N-1,\bO)}(\Gr_{\SO_N})$.
To explain this, let $\on{G}:=\SO_{N-1}\times\SO_N$ and
$\on{H}:=\SO_{N-1}$. We view $\on{H}$ as a block-diagonal subgroup of $\on{G}$
and put $X=\on{G}/\on{H}$.
Then loosely speaking we have $D_{\SO(N-1,\bO)}(\Gr_{\SO_{N}})\simeq
D\big(\SO(N-\nolinebreak1,\bO)\backslash\SO(N,\bF)/\SO(N,\bO)\big)\simeq
D\big(\on{G}(\bO)\backslash\on{G}(\bF)/\on{H}(\bF)\big)\simeq
D\big(\on{G}(\bO)\backslash X(\bF)\big)\simeq D_{\on{G}(\bO)}(X(\bF))$.
On the other hand, note that $\on{G}^\vee=\SO(V_0)\times\Sp(V_1)$.
We consider a graded $\on{G}^\vee$-module
$V^\vee_X:=(V_0\otimes V_1)[1]$ (we view $V^\vee_X$ as an
  {\em odd} vector space placed in cohomological degree $-1$).
Hence, the equivalence $\Phi$ takes the form
$D\Coh(V^\vee_X/\on{G}^\vee)\simeq D_{\on{G}(\bO)}(X(\bF))$.

\subsection{Conjectural Iwahori-equivariant version}
Similarly to~\cite[\S1.4]{bfgt} we propose the following conjecture.
Let $\Ffl_1$ denote the variety of complete self-orthogonal flags in $V_1$, and let
$\Ffl_0$ denote a connected component of the variety of complete self-orthogonal flags in $V_0$
(there are two canonically isomorphic connected components, and we choose one).
We consider a dg-scheme with trivial differential
\[\on{H}_\osp:=(V_0\otimes V_1)[1]\times\Ffl_0\times\Ffl_1.\]
Here we view $V_0\otimes V_1$ as an {\em odd} vector space, so that the functions on
$(V_0\otimes V_1)[1]$ (with grading disregarded) form really a symmetric (infinite-dimensional)
algebra, not an exterior algebra. We will write $A$ for an element of
$V_0\otimes V_1\cong\Hom(V_0,V_1)$, and $A^t$ for the adjoint operator in $\Hom(V_1,V_0)$.
We will also write
$F_i=(F_i^{(1)}\subset F_i^{(2)}\subset\ldots\subset F_i^{(\dim V_i)}=V_i)$ for an element of
$\Ffl_i,\ i=0,1$.

We define the {\em orthosymplectic Steinberg scheme} to be a dg-subscheme $\on{St}_\osp$ of
$\on{H}_\osp$ cut out by the equations saying that the flag $F_0$ is stable under the composition
$A^tA$ and  the flag $F_1$ is stable under the composition
$AA^t$. Thus the orthosymplectic Steinberg scheme is a shifted variety of triples:
\[
\on{St}_\osp=\{(A,F_0,F_1)\in \on{H}_\osp\mid
A^tA(F^{(r)}_0)\subseteq F^{(r)}_0\en \& \en
AA^t(F^{(r)}_1)\subseteq F^{(r)}_1,\ \forall r\}.
\]
Let $\bI_{N-1}\subset\SO(N-1,\bO)$ (resp.\ $\bI_N\subset\SO(N,\bO)$) be an Iwahori subgroup
and let $\Fl_{\SO_N}:=\SO(N,\bF)/\bI_N$ be the affine flag variety.
Let $D^b_{\bI_{N-1}}(\Fl_{\SO_N})$ be the bounded $\bI_{N-1}$-equivariant constructible derived
category of $\Fl_{\SO_N}$. We propose the following
\begin{conj} There exists an equivalence of triangulated categories
\[
D^{\SO(V_0)\times\Sp(V_1)}  \on{Coh}(\on{St}_\osp)
\cong D^b_{\bI_{N-1}}(\Fl_{\SO_N}).
\]
\end{conj}

This conjecture would give an alternative proof of~Theorem~\ref{ic stalk} expressing the stalks of
$\SO(N-1,\bO)$-equivariant IC-sheaves on $\Gr_{\SO_N}$ in terms of {\em orthosymplectic
  Kostka polynomials} introduced in~\S\ref{kostka} as a particular case of general
construction due to D.~Panyushev~\cite{p}.

\subsection{Gaiotto conjectures}
One may wonder if there is a geometric realization of representations of {\em nondegenerate}
orthosymplectic supergroups. It turns out that such a realization exists (conjecturally) for
the categories of integrable representations of {\em quantized} type $D$ orthosymplectic algebras
$U_q(\osp(2k|2l))$.
First of all, similarly to the
classical Kazhdan-Lusztig equivalence, it is expected that
$U_q(\osp(2k|2l))\on{-mod}\cong\on{KL}_c(\widehat\osp(2k|2l))$, where $q=\exp(\pi\sqrt{-1}/c)$,
and $\on{KL}_c(\widehat\osp(2k|2l))$ stands for the derived category of
$\SO(2k,\bO)\times\Sp(2l,\bO)$-equivariant
$\widehat\osp(2k|2l)$-modules of the central charge corresponding to the invariant bilinear form
$(X,Y)=c\cdot\on{sTr}(XY)-\frac12\Killing_{\osp(2k|2l)}(X,Y)$ on $\osp(2k|2l)$.
Second, it is expected that
the category $\on{KL}_c(\widehat\osp(2n|2n))$ is equivalent to the $q$-monodromic
$\SO(2n,\bO)$-equivariant derived constructible category of the complement $\CL^\bullet_{2n+1}$ of the
zero section of the determinant line bundle on $\Gr_{\SO_{2n+1}}$, and this equivalence takes the
standard $t$-structure of $\on{KL}_c(\widehat\osp(2n|2n))$ to the perverse $t$-structure.

Further, it is expected that the category $\on{KL}_c(\widehat\osp(2n|2n-2))$ is equivalent
to the $q$-monodromic $\SO(2n-1,\bO)$-equivariant derived constructible category of the
complement of the zero section $\CL^\bullet_{2n}$ of the determinant line bundle on $\Gr_{\SO_{2n}}$,
and this equivalence takes the
standard $t$-structure of $\on{KL}_c(\widehat\osp(2n|2n-2))$ to the perverse $t$-structure.
For other values of $(2k|2l)$ the situation depends on the dichotomy $2k-1<2l$ or
$2k-1>2l$.
In case $2k-1<2l$ it is expected that $\on{KL}_c(\widehat\osp(2k|2l))$ is equivalent to the
$q$-monodromic $\SO(2k,\bO)$-equivariant derived constructible category of
$\CL^\bullet_{2l+1}$ with certain Whittaker conditions, cf.~\S\ref{gaiotto} for more details.
In case $2k\nolinebreak-\nolinebreak 1>\nolinebreak 2l$ it is expected that
$\on{KL}_c(\widehat\osp(2k|2l))$ is equivalent
to the $q$-monodromic $\SO(2l+1,\bO)$-equivariant derived constructible category of
$\CL^\bullet_{2k}$ with certain Whittaker conditions, cf.~\S\ref{gaiotto} for more details.
In particular, the special cases $k=0$ or $l=0$ of this conjecture follow from the
Fundamental Local Equivalence of the geometric Langlands program, see~\cite[\S2]{bfgt}.

In the case $(2k|2l)=(4|2)$, each connected component of $\Gr_{\SO_4}$ is isomorphic to
$\Gr_{\on{SL}_2}\times\Gr_{\on{SL}_2}$, so that the Picard group of each connected component is
generated by {\em two} determinant line bundles, and we have one extra degree of freedom
in twisting parameters. It is expected that the corresponding
categories of equivariant monodromic perverse sheaves are equivalent to the Kazhdan-Lusztig
categories for the affine Lie superalgebras $D(2,1;\alpha)^{(1)}$, cf.~Remark~\ref{d21}.

\subsection{Acknowledgments}
We are grateful to A.~Berezhnoy, R.~Bezrukavnikov, I.~Entova-Aizenbud, P.~Etingof, B.~Feigin,
D.~Gaiotto, D.~Gaitsgory, D.~Leites, I.~Motorin,
Y.~Sakellaridis, V.~Serganova, A.~Venkatesh and E.~B.~Vinberg for very useful discussions.
Above all, we are indebted to V.~Ginzburg: it should be clear from the above that this paper
is but an outgrowth of the project initiated by him, worked out by the tools developed by him.
A.B.\ was partially supported by NSERC. M.F.\ was partially funded within the framework of the HSE
University Basic Research Program and the Russian Academic Excellence Project `5-100'.

\section{A coherent realization of $D^b_{\SO(N-1,\bO)}(\Gr_{\SO_N})$}
\label{2}\nopagebreak
\subsection{Orthogonal and symplectic Lie algebras}
\label{setup}
In both cases~\ref{OSE}(a,b) the tensor product space $V_0\otimes V_1$ is equipped with a
nondegenerate skew-symmetric bilinear form $(,)\otimes\langle,\rangle$. It is preserved by the
action of the group $\SO(V_0)\times\Sp(V_1)$. The corresponding moment map is described as follows.

Our nondegenerate bilinear forms on $V_0,V_1$ define identifications
$V_0\cong V_0^*,\ V_1\cong V_1^*$. In particular, $V_0\otimes V_1$ is identified with
$V_0^*\otimes V_1=\Hom(V_0,V_1)$. Given $A\in\Hom(V_0,V_1)$ we have the adjoint operator
$A^t\in\Hom(V_1,V_0)$. We have the moment maps
\[\bq_0\colon V_0\otimes V_1\to\fso(V_0)^*,\ A\mapsto A^tA,\ \operatorname{and}\
\bq_1\colon V_0\otimes V_1\to\fsp(V_1)^*,\ A\mapsto AA^t,\]
where we make use of the identification $\fso(V_0)\cong\fso(V_0)^*$
(resp.\ $\fsp(V_1)\cong\fsp(V_1)^*$)
via the trace form (resp.\ {\em negative} trace form) of the defining representation.
Note also that the complete moment map $(\bq_0,\bq_1)$ coincides with the ``square''
(half-self-supercommutator) map on the odd part $\bg_{\bar1}$ of the orthosymplectic Lie
superalgebra $\bg$. We define the odd nilpotent cone $\CN_{\bar1}\subset V_0\otimes V_1$
as the {\em reduced} subscheme cut out by
the condition of nilpotency of $A^tA$ (equivalently, by the condition of nilpotency of $AA^t$).

We choose Cartan subalgebras $\ft_0\subset\fso(V_0)$ and $\ft_1\subset\fsp(V_1)$.
We choose a basis $\varepsilon_1,\ldots,\varepsilon_n$ in $\ft_0^*$ such that the Weyl group
$W_0=W(\fso(V_0),\ft_0)$ acts by permutations of basis elements and by the sign changes
of an even number of basis elements, and the roots of $\fso(V_0)$ are given by $\{\pm\eps_i\pm\eps_j,\ i\neq j\}$. We set $\Sigma_0=\ft_0^*\dslash W_0$.
We also choose a basis $\delta_1,\ldots,\delta_n$ in $\ft_1^*$ in the odd case
(resp.\ $\delta_1,\ldots,\delta_{n-1}$ in the even case) such that the Weyl group
$W_1=W(\fsp(V_1),\ft_1)$ acts by permutations of basis elements and by the sign changes
of basis elements, and the roots of $\fsp(V_1)$ are given by $\{\pm\delta_i\pm\delta_j,\ i\neq j;\ \pm2\delta_i\}$. We set $\Sigma_1=\ft_1^*\dslash W_1$.

In the odd case we identify $\ft_0^*\cong\ft_1^*,\ \varepsilon_i\mapsto\delta_i$,
and this identification gives rise
to a two-fold cover $\varPi_{01}\colon \Sigma_0\to\Sigma_1$. Similarly, in the even case
we identify $\ft_1^*$ with a hyperplane in $\ft_0^*,\ \delta_i\mapsto\varepsilon_i$,
and this identification gives rise
to a closed embedding $\varPi_{10}\colon \Sigma_1\hookrightarrow\Sigma_0$.

Recall (see e.g.~\cite[\S\S2.1,2.6]{bf}) that $\Sigma_0$ is embedded as a Kostant slice
into the open set of regular elements $(\fso(V_0)^*)^\reg\subset\fso(V_0)^*$, and $\Sigma_1$
is embedded into $(\fsp(V_1)^*)^\reg$. Furthermore, these slices $\Sigma_0,\Sigma_1$ carry the
universal centralizer sheaves of abelian Lie algebras $\fz_0,\fz_1$. Given an $\SO(V_0)$-module $V$
(resp.\ an $\Sp(V_1)$-module $V'$), we have the corresponding graded
$\Gamma(\Sigma_0,\fz_0)$-module $\kappa_0(V)$ (resp.\ the
$\Gamma(\Sigma_1,\fz_1)$-module $\kappa_1(V')$)
(the {\em Kostant functor} of {\em loc.~cit.}).
Since the universal enveloping algebra $U(\fz_0)$ (resp.\ $U(\fz_1)$) is identified in
{\em loc.\ cit.} with the sheaf of functions on the tangent bundle $T\Sigma_0$ (resp.\
$T\Sigma_1$), we will use the same notation $\kappa_0(V),\kappa_1(V')$ for the corresponding
coherent sheaves on $T\Sigma_0, T\Sigma_1$.
Finally, according to the previous paragraph, we have the morphisms
$d\varPi_{01}\colon T\Sigma_0\to T\Sigma_1$ in the odd case and
$d\varPi_{10}\colon T\Sigma_1\to T\Sigma_0$ in the even case.

We choose Borel subalgebras $\ft_0\subset\fb_0\subset\fso(V_0)$ corresponding to the
choice of positive roots $R_0^+=\{\varepsilon_i\pm\varepsilon_j,\ i<j\}$
and $\ft_1\subset\fb_1\subset\fsp(V_1)$ corresponding to the choice of positive roots
$R_1^+=\{\delta_i\pm\delta_j,\ i<j;\ 2\delta_i\}$. We set $\rho_0=\frac12\sum_{\alpha\in R_0^+}\alpha$
and $\rho_1=\frac12\sum_{\alpha\in R_1^+}\alpha$.
We denote by $\Lambda_0$ (resp.\ $\Lambda_1$) the weight
lattice of $\SO(V_0)$ (resp.\ of $\Sp(V_1)$). We denote by $\Lambda_0^+\subset\Lambda_0$
(resp.\ $\Lambda_1^+\subset\Lambda_1$)
the monoids of dominant weights. For $\lambda\in\Lambda_0^+$ (resp.\ $\lambda\in\Lambda_1^+$)
we denote by $V_\lambda$ the irreducible representation of $\SO(V_0)$ (resp.\ of $\Sp(V_1)$)
with highest weight $\lambda$.

\medskip

In what follows $\SO(V_0)$ will play the role of the Langlands dual group of
$\SO_{N-1}$ (resp.\ of $\SO_N$) in the odd (resp.\ even) case, while $\Sp(V_1)$ will play the role
of the Langlands dual group of $\SO_N$ (resp.\ of $\SO_{N-1}$) in the odd (resp.\ even) case.
For this reason we will need various claims that are formulated and even proved similarly
in the odd/even cases up to replacing symplectic groups with special orthogonal groups
(especially in~\S\ref{IT}).
In order to save space and not to duplicate numerous claims, we introduce the following
`blinking' notation. We set $G_1=\Sp(V_1),\ G_0=\SO(V_0)$ (not to be confused with $\sG_{\bar0}$!),
and let $(\sfb,\sfs)=(1,0)$ (resp.\ $(\sfb,\sfs)=(0,1)$) in the odd (resp.\ even) case.
Then $G_\sfb=\SO_N^\vee$ is the group of bigger dimension, and $G_\sfs=\SO_{N-1}^\vee$ is the
group of smaller dimension. Accordingly, we set $\fg_1=\fsp(V_1),\ \fg_0=\fso(V_0)$
(not to be confused with $\bg_{\bar1},\bg_{\bar0}$!), and get $\dim\fg_\sfb>\dim\fg_\sfs$.
Similarly, we have $\dim V_\sfb\geq\dim V_\sfs$ and $\varPi_{\sfs\sfb}\colon\Sigma_\sfs\to\Sigma_\sfb$
(but we do {\em not} have $\varPi_{\sfb\sfs}$), etc.

\subsection{The main theorem}
\label{osp}
Recall the orthosymplectic Lie superalgebra $\bg=\osp(V_0|V_1)$ of~\S\ref{OSE}.
We consider the dg-algebra\footnote{We view $\bg_{\bar1}$ as an {\em odd} vector space, so that
  $\Sym(\bg_{\bar1}[-1])$ (with grading disregarded) is really a symmetric
  (infinite-dimensional) algebra, not an exterior algebra.} $\fG^\bullet=\Sym(\bg_{\bar1}[-1])$
with trivial differential,
and the triangulated category $D^{\sG_{\bar0}}_\perf(\fG^\bullet)$ obtained by localization (with
respect to quasi-isomorphisms) of the category of perfect $\sG_{\bar0}$-equivariant
dg-$\fG^\bullet$-modules. We also consider the corresponding category $SD^{\sG_{\bar0}}_\perf(\fG^\bullet)$
with coefficients in super vector spaces. Since $\fG^\bullet$ is super-commutative, we have
a symmetric monoidal structure $\otimes_{\fG^\bullet}$ on the category $SD^{\sG_{\bar0}}_\perf(\fG^\bullet)$.

The action of the central element $(\Id_{V_0},-\Id_{V_1})\in\sG_{\bar0}$ on an object of
$D^{\sG_{\bar0}}_\perf(\fG^\bullet)$ equips this object with an extra $\BZ/2\BZ$-grading, and thus
defines a fully faithful functor
$D^{\sG_{\bar0}}_\perf(\fG^\bullet)\to SD^{\sG_{\bar0}}_\perf(\fG^\bullet)$ of ``superization'',
such that its essential image is
closed under the monoidal structure $\otimes_{\fG^\bullet}$. This defines the
monoidal structure $\otimes_{\fG^\bullet}$ on the category $D^{\sG_{\bar0}}_\perf(\fG^\bullet)$.

We consider the following complex $H^\bullet$ of {\em odd} vector spaces living in degrees $0,1\colon
\bg_{\bar1}\stackrel{\Id}{\longrightarrow}\bg_{\bar1}$.
We define the Koszul complex $K^\bullet$ as the symmetric algebra $\Sym(H^\bullet)$. The degree
zero part \[K^0=\Lambda(V_0\otimes V_1)=:\Lambda\] (as a vector space,
with a super-structure disregarded). We turn $K^\bullet$ into a  dg-$\fG^\bullet-\Lambda$-bimodule
by letting $\fG^\bullet$ act by multiplication, and $\Lambda$ by differentiation.
Note that $K^\blt$ is quasi-isomorphic to $\BC$ in degree 0 as a complex of vector spaces,
but {\em not} as a dg-$\fG^\bullet-\Lambda$-bimodule.

We consider the derived category $D_\fid^{\sG_{\bar0}}(\Lambda)$ of finite dimensional complexes of
$\sG_{\bar0}\ltimes\Lambda$-modules. If we remember the super-structure of $\Lambda$, we obtain
the corresponding category of super dg-modules $SD_\fid^{\sG_{\bar0}}(\Lambda)=D^b\Rep(\ul\sG)$.
We have the Koszul equivalence functors
\[\varkappa\colon D_\fid^{\sG_{\bar0}}(\Lambda)\iso D^{\sG_{\bar0}}_\perf(\fG^\bullet),\
D\Rep(\ul\sG)=SD_\fid^{\sG_{\bar0}}(\Lambda)\iso SD^{\sG_{\bar0}}_\perf(\fG^\bullet),\
\CM\mapsto K^\bullet\otimes_\Lambda\CM.\]
The Koszul equivalence
$\varkappa\colon D^b\Rep(\ul\sG)\iso SD^{\sG_{\bar0}}_\perf(\fG^\bullet)$ is monoidal
with respect to the usual tensor structure on the LHS and $\otimes_{\fG^\bullet}$ on the RHS.

The action of $(\Id_{V_0},-\Id_{V_1})\in\sG_{\bar0}$ gives rise to a
fully faithful ``superization'' functor
$D_\fid^{\sG_{\bar0}}(\Lambda)\to SD_\fid^{\sG_{\bar0}}(\Lambda)=D^b\Rep(\ul\sG)$ with the essential image
closed under the tensor structure.
This defines the tensor structure on $D_\fid^{\sG_{\bar0}}(\Lambda)$ such that the Koszul
equivalence $\varkappa\colon D_\fid^{\sG_{\bar0}}(\Lambda)\iso D^{\sG_{\bar0}}_\perf(\fG^\bullet)$ is
monoidal.

\medskip

Recall the quadratic moment maps
$\fso(V_0)^*\stackrel{\bq_0}{\longleftarrow}V_0\otimes V_1\stackrel{\bq_1}{\longrightarrow}\fsp(V_1)^*$
of~\S\ref{setup}. They give rise to homomorphisms
\[\Sym\!\big(\fso(V_0)[-2]\big)\xrightarrow{\bq_0^*}\fG^\bullet
=\Sym\!\big(\Pi(V_0\otimes V_1)[-1]\big)\xleftarrow{\bq_1^*}\Sym\!\big(\fsp(V_1)[-2]\big)\]
and to the corresponding induction functors
\[D^{\SO(V_0)}_\perf\Big(\Sym\!\big(\fso(V_0)[-2]\big)\Big)\xrightarrow{\bq_0^*}
D^{\sG_{\bar0}}_\perf(\fG^\bullet)\xleftarrow{\bq_1^*}
D^{\Sp(V_1)}_\perf\Big(\Sym\!\big(\fsp(V_1)[-2]\big)\Big).\]
Thus the category $D^{\sG_{\bar0}}_\perf(\fG^\bullet)$ acquires a module structure over the
monoidal category $D^{\SO(V_0)}_\perf\Big(\Sym\!\big(\fso(V_0)[-2]\big)\Big)\otimes
D^{\Sp(V_1)}_\perf\Big(\Sym\!\big(\fsp(V_1)[-2]\big)\Big)$. Recall the `blinking' notation
of~\S\ref{setup}, so that the latter monoidal category is denoted
$D^{G_\sfs}_\perf\big(\Sym(\fg_\sfs[-2])\big)\otimes D^{G_\sfb}_\perf\big(\Sym(\fg_\sfb[-2])\big)$.
Also recall the equivalences
\[D^{G_\sfs}_\perf\big(\Sym(\fg_\sfs[-2])\big)\xrightarrow[\boldsymbol\beta]{\sim}
D^b_{\SO(N-1,\bO)}(\Gr_{\SO_{N-1}}),\]
\[D^{G_\sfb}_\perf\big(\Sym(\fg_\sfb[-2])\big)\xrightarrow[\boldsymbol\beta]{\sim}
D^b_{\SO(N,\bO)}(\Gr_{\SO_N})\] of~\cite[Theorem 5]{bf}.

\medskip

Finally recall the odd nilpotent cone $\CN_{\bar1}\subset V_0\otimes V_1$ of~\S\ref{setup}.
We denote by $D^{\sG_{\bar0}}_\perf(\fG^\bullet)_{\CN_{\bar1}}$ the full subcategory of
$D^{\sG_{\bar0}}_\perf(\fG^\bullet)$ formed by
complexes with cohomology set-theoretically supported at $\CN_{\bar1}$. We also
denote by $D^{\on{comp}}_{\SO(N-1,\bO)}(\Gr_{\SO_N})$ the full subcategory of
$D^b_{\SO(N-1,\bO)}(\Gr_{\SO_N})$ formed by the objects compact as the objects of the {\em unbounded}
category $D_{\SO(N-1,\bO)}(\Gr_{\SO_N})$.

\bigskip

Our goal is the following

\begin{thm}
\label{main}
  \textup{(a)} There exists an equivalence of triangulated categories
  $\Phi\colon D^{\sG_{\bar0}}_\perf(\fG^\bullet)\iso D^b_{\SO(N-1,\bO)}(\Gr_{\SO_N})$
  commuting with the left convolution action of the monoidal
  spherical Hecke category $D^{G_\sfs}_\perf\big(\Sym(\fg_\sfs[-2])\big)\cong D^b_{\SO(N-1,\bO)}(\Gr_{\SO_{N-1}})$
  and with the right convolution action of the monoidal
  spherical Hecke category $D^{G_\sfb}_\perf\big(\Sym(\fg_\sfb[-2])\big)\cong D^b_{\SO(N,\bO)}(\Gr_{\SO_N})$.

\textup{(b)} The composed equivalence
  \[\Phi\circ\varkappa\colon
  D_\fid^{\sG_{\bar0}}(\Lambda)\iso D^b_{\SO(N-1,\bO)}(\Gr_{\SO_N})\]
  is exact with respect to the tautological $t$-structure on
  $D_\fid^{\sG_{\bar0}}(\Lambda)$ and the perverse $t$-structure on $D^b_{\SO(N-1,\bO)}(\Gr_{\SO_N})$.

  \textup{(c)} This equivalence is monoidal with respect to the tensor structure on
  $D_\fid^{\sG_{\bar0}}(\Lambda)$ and the fusion $\star$ on $D^b_{\SO(N-1,\bO)}(\Gr_{\SO_N})$.

  \textup{(d)} The equivalence of \textup{(b)} extends to a monoidal equivalence from
  $SD_\fid^{G_{\bar0}}(\Lambda)=D^b\Rep(\ul\sG)$ to the equivariant derived constructible category
  with coefficients in super vector spaces $SD^b_{\SO(N-1,\bO)}(\Gr_{\SO_N})$.

  \textup{(e)} The equivariant derived category $D^b_{\SO(N-1,\bO)}(\Gr_{\SO_N})$ is equivalent to the
  bounded derived category of the abelian category $\Perv_{\SO(N-1,\bO)}(\Gr_{\SO_N})$.

  \textup{(f)} $\Phi$ induces an equivalence
  $D^{\sG_{\bar0}}_\perf(\fG^\bullet)_{\CN_{\bar1}}\iso D^{\on{comp}}_{\SO(N-1,\bO)}(\Gr_{\SO_N})$.
  In particular, $\Phi$ extends to an equivalence
  \[\QCoh_{\CN_{\bar1}}\big(\Pi(V_0\otimes V_1)[1]/\sG_{\bar0}\big)\iso D_{\SO(N-1,\bO)}(\Gr_{\SO_N}).\]
  Also, a sheaf $\CF\in D^b_{\SO(N-1,\bO)}(\Gr_{\SO_N})$ lies in $D^{\on{comp}}_{\SO(N-1,\bO)}(\Gr_{\SO_N})$
  iff $\dim H^\bullet_{\SO(N-1,\bO)}(\Gr_{\SO_N},\CF)<\infty$.
\end{thm}

The proof will be given in~\S\ref{monoidal} after some preparations
in~\S\S\ref{deeq}--\ref{nilpotent}.

\subsection{$\SO(N-1,\bO)$-orbits in $\Gr_{\SO_N}$}
\label{orb}
The following lemma is well known to the experts; we learned it from Y.~Sakellaridis.

\begin{lem}
  \label{sake}
  There is a natural bijection between the set of $\SO(N-1,\bO)$-orbits on $\Gr_{\SO_N}$
  and the monoid of dominant coweights of $\SO_{N-1}\times\SO_N$.
\end{lem}

\begin{proof}
We consider the block-diagonal embedding $\SO_{N-1}\hookrightarrow\SO_{N-1}\times\SO_N$.
Then the set of orbits of $\SO(N-1,\bO)$ in $\Gr_{\SO_N}$ is in natural bijection with the set of
orbits of $\SO(N-1,\bF)$ in $\Gr_{\SO(N-1,\bO)}\times\Gr_{\SO(N,\bO)}$. Furthermore,
$X=(\SO_{N-1}\times\SO_N)/\SO_{N-1}$ is a homogeneous spherical variety of
$\on{G}:=\SO_{N-1}\times\SO_N$,
and the latter set of orbits is identified with the monoid $\Lambda_X^+$ of
$\on{G}$-invariant valuations on $\BC(X)$.
The proof goes back to~\cite[\S8]{lv}; for a modern exposition
see e.g.~\cite[Theorem~8.2.9]{gn}. Furthermore, the monoid $\Lambda_X^+$ coincides with
the monoid of dominant weights of the Gaitsgory-Nadler group $\on{G}_X^\vee$.
In our case $\on{G}_X^\vee$ coincides with the Langlands dual group
$\on{G}^\vee=\SO_{N-1}^\vee\times\SO_N^\vee$.

Indeed, the corresponding rational cone $\Lambda_{X,\BQ}^+$ can be computed from
the Luna diagram (aka Luna spherical system) of our spherical variety. In our case,
the Luna diagram is described e.g.\ in~\cite[(46),(50)]{bp}, and it follows that all the
simple roots of $\on{G}$ are spherical roots for $X$, i.e.\
the little Weyl group $W_X$ coincides with the Weyl group $W_0\times W_1$ of
$\SO_{N-1}\times\SO_N$. Hence $\Lambda_{X,\BQ}^+=\Lambda_{0,\BQ}^+\times\Lambda_{1,\BQ}^+$
(notation of~\S\ref{setup}).
In order to identify the monoid of dominant weights inside the rational cone
it suffices to check that the stabilizer in $\SO_{N-1}$ of a general point in the flag variety
of $\on{G}$ is trivial.

In the odd case~\ref{OSE}(a) we choose a basis $v_1,v_2,\ldots,v_{2n},v_{2n+1}$ in a vector
space $V$ equipped with symmetric bilinear form such that $v_{2n+1},v_{2n},\ldots,v_2,v_1$ is the
dual basis, and $\SO_{2n}\subset\SO_{2n+1}$ is the stabilizer of $v_{n+1}$.
We define a complete isotropic flag $U_1\subset U_2\subset\ldots\subset U_n\subset(\BC v_{n+1})^\perp$
and a complete isotropic flag $U'_1\subset U'_2\subset\ldots\subset U'_n\subset V$ as follows:
\[U_i:=\BC v_1\oplus\ldots\oplus\BC v_i,\
U'_i:=\BC v'_{2n+1}\oplus\ldots\oplus\BC v'_{2n+2-i},\]
where $v'_{2n+2-i}=v_{2n+2-i}-v_{n+1}-\frac12(v_1+v_2+\ldots+v_n)$.
It is immediate to see that $\on{Stab}_{\SO_{N-1}}(U_\bullet,U'_\bullet)$ is trivial.
In the even case~\ref{OSE}(b) the argument is similar.
\end{proof}

Note that in the odd case~\ref{OSE}(a),
$\SO_{N-1}^\vee\cong\SO(V_0),\ \SO_N^\vee\cong\Sp(V_1)$, while in the even case~\ref{OSE}(b),
$\SO_{N-1}^\vee\cong\Sp(V_1),\ \SO_N^\vee\cong\SO(V_0)$.
We will use another construction of bijection
$\Lambda_0^+\times\Lambda_1^+\cong\SO(N-1,\bO)\backslash\Gr_{\SO_N}$ (presumably it coincides
with the bijection of~Lemma~\ref{sake}, but we did not check this).
In the blinking notation of~\S\ref{setup}, given dominant coweights
$\lambda_\sfs\in\Lambda_\sfs^+,\ \lambda_\sfb\in\Lambda_\sfb^+$ we denote by
$\ol\Gr{}^{\lambda_\sfs}_{\SO_{N-1}}\wtimes\ol\Gr{}^{\lambda_\sfb}_{\SO_N}
\stackrel{\bm}{\longrightarrow}\Gr_{\SO_N}$ the convolution diagram of spherical Schubert varieties.
The convolution morphism $\bm$ is clearly $\SO(N-1,\bO)$-equivariant, so there is a well defined
$\SO(N-1,\bO)$-orbit in $\Gr_{\SO_N}$ open in the image of $\bm$. We will denote this orbit
$\BO^{\lambda_\sfs}_{\lambda_\sfb}\subset\Gr_{\SO_N}$.

\begin{lem}
  \label{sake?}
  The map $(\lambda_\sfs,\lambda_\sfb)\mapsto\BO^{\lambda_\sfs}_{\lambda_\sfb}$ is a bijection
  \[\Lambda_\sfs^+\times\Lambda_\sfb^+\iso\SO(N-1,\bO)\backslash\Gr_{\SO_N}.\]
\end{lem}

\begin{proof}
  We start with a similar parametrization of the set of $\GL(N-1,\bO)$-orbits in $\Gr_{\GL_N}$
  or equivalently, of the set of $\GL(N-1,\bF)$-orbits in $\Gr_{\GL_{N-1}}\times\Gr_{\GL_N}$.
  We choose a basis $e_1,\ldots,e_N$
  in the defining representation $\BC^N$ of $\GL_N$, so that the defining representation of
  $\GL_{N-1}$ is spanned by $e_1,\ldots,e_{N-1}$. Then one can choose the following set of
  representatives of $\GL(N-1,\bF)$-orbits in $\Gr_{\GL_{N-1}}\times\Gr_{\GL_N}$, as follows from the
  proof of~\cite[Proposition 8]{fgt}.
  Recall that $\Gr_{\GL_{N-1}}$ (resp.\ $\Gr_{\GL_N}$) is the moduli space of lattices
  in $\bF\otimes\BC^{N-1}$ (resp.\ in $\bF\otimes\BC^N$). Given signatures (non-increasing
  sequences of integers) \[\bmu=(\mu_1\geq\mu_2\geq\ldots\geq\mu_{N-1}),\
  \bnu=(\nu_1\geq\nu_2\geq\ldots\geq\nu_N)\] we consider the lattices
  \[L'_\bmu:=\bO t^{\mu_1}e_1\oplus\ldots\oplus\bO t^{\mu_{N-1}}e_{N-1}\subset\bF\otimes\BC^{N-1},\]
  \[L_\bnu:=\bO t^{-\nu_1}(e_1+e_N)\oplus\ldots\oplus\bO t^{-\nu_{N-1}}(e_{N-1}+e_N)\oplus \bO t^{-\nu_N}e_N
  \subset\bF\otimes\BC^N.\]
  Such pairs form a complete set of representatives of $\GL(N-1,\bF)$-orbits on
  $\Gr_{\GL_{N-1}}\times\Gr_{\GL_N}$ as $\bmu$ (resp.\ $\bnu$) runs through the set of all length $N-1$
  (resp.\ length $N$) signatures. Hence the following set of lattices in $\bF\otimes\BC^N$
  \[\{L_{\bmu,\bnu}:=\bO(t^{-\mu_1-\nu_1}e_1+t^{-\nu_1}e_N)\oplus\ldots\oplus
  \bO(t^{-\mu_{N-1}-\nu_{N-1}}e_{N-1}+t^{-\nu_{N-1}}e_N)\oplus\bO t^{-\nu_N}e_N\}\]
  is a complete set of representatives of $\GL(N-1,\bO)$-orbits in $\Gr_{\GL_N}$.
  Clearly, $L_{\bmu,\bnu}$ lies in the image of the convolution morphism
  $\bm\colon \ol\Gr{}^{\bmu}_{\GL_{N-1}}\wtimes\ol\Gr{}^{\bnu}_{\GL_N}\to\Gr_{\GL_N}$,
  and the orbit $\BO_{\bmu,\bnu}:=\GL(N-1,\bO)\cdot L_{\bmu,\bnu}$ is open in the image of $\bm$.

We return back to special orthogonal groups, and realize $\SO_M$ as the connected component
of invariants of an involution of $\GL_M$. Accordingly, $\Gr_{\SO_M}$ is a union of connected
components of the fixed point set of the corresponding involution $\varsigma$ of $\Gr_{\GL_M}$.
It follows that any $\SO(N-1,\bO)$-orbit in $\Gr_{\SO_N}$ is a connected component of the
fixed point set $\BO_{\bmu,\bnu}^\varsigma$ of an appropriate $\GL(N-1,\bO)$-orbit in $\Gr_{\GL_N}$.
Recall that the convolution diagram $\ol\Gr{}^{\bmu}_{\GL_{N-1}}\wtimes\ol\Gr{}^{\bnu}_{\GL_N}$
is a fibre bundle over $\ol\Gr{}^{\bmu}_{\GL_{N-1}}$ with fibers isomorphic to $\ol\Gr{}^{\bnu}_{\GL_N}$,
and the convolution morphism
$\bm\colon \ol\Gr{}^{\bmu}_{\GL_{N-1}}\wtimes\ol\Gr{}^{\bnu}_{\GL_N}\to\ol\BO_{\bmu,\bnu}$ is
a birational isomorphism (more precisely, $\bm$ is an isomorphism over
$\BO_{\bmu,\bnu}\subset\ol\BO_{\bmu,\bnu}$). It follows that for a connected component
$\BO_{\bmu,\bnu}^{\varsigma,0}$ of
$\BO_{\bmu,\bnu}^\varsigma$ there are appropriate irreducible components of the fixed point sets
$(\ol\Gr{}^{\bmu}_{\GL_{N-1}})^{\varsigma,0}\subset(\ol\Gr{}^{\bmu}_{\GL_{N-1}})^\varsigma,\
(\ol\Gr{}^{\bnu}_{\GL_N})^{\varsigma,0}\subset(\ol\Gr{}^{\bnu}_{\GL_N})^\varsigma$ such that
$\bm$ induces a birational isomorphism to the closure $\ol\BO{}_{\bmu,\bnu}^{\varsigma,0}$ from the fibre
bundle over $(\ol\Gr{}^{\bmu}_{\GL_{N-1}})^{\varsigma,0}$ with fibers isomorphic to
$(\ol\Gr{}^{\bnu}_{\GL_N})^{\varsigma,0}$. However, any irreducible component
$(\ol\Gr{}^{\bmu}_{\GL_{N-1}})^{\varsigma,0}$ (resp.\ $(\ol\Gr{}^{\bnu}_{\GL_N})^{\varsigma,0}$) coincides with
$\ol\Gr{}^{\lambda_\sfs}_{\SO_{N-1}}$ (resp.\ with $\ol\Gr{}^{\lambda_\sfb}_{\SO_N}$) for appropriate
coweights $\lambda_\sfs,\lambda_\sfb$.

The lemma is proved.
\end{proof}

We denote by $\IC^{\lambda_\sfs}_{\lambda_\sfb}\in\Perv_{\SO(N-1,\bO)}(\Gr_{\SO_N})$ the intermediate extension
of the constant local system on $\BO^{\lambda_\sfs}_{\lambda_\sfb}$. We will denote $\IC^0_0$ by $E_0$ for
short.

\begin{lem}
  \label{connected}
  Any $\SO(N-1,\bO)$-equivariant irreducible perverse sheaf on $\Gr_{\SO_N}$ is of the form
  $\IC^{\lambda_\sfs}_{\lambda_\sfb}$.
\end{lem}

\begin{proof}
  We have to check that the stabilizer in $\SO(N-1,\bO)$ of a point in
  $\Gr_{\SO_N}$ is connected. Equivalently,
  we have to check that the stabilizer in $\SO(N-1,\bF)$ of a point in
  $\Gr_{\SO_{N-1}}\times\Gr_{\SO_N}$ is connected.
  It follows from the proof of~Lemma~\ref{sake?} that the following list of pairs
  $(L'_\bmu,L_\bnu)$ forms a complete set of representatives of $\SO(N-1,\bF)$-orbits
  in $\Gr_{\SO_{N-1}}\times\Gr_{\SO_N}$ (for an appropriate choice of an involution of $\GL_M$
  producing $\SO_M$ as the connected component of the fixed point set):

  In the odd case~\ref{OSE}(a)
  \[\bnu=(\nu_1\geq\nu_2\geq\ldots\geq\nu_n\geq0\geq-\nu_n\geq-\nu_{n-1}\ldots\geq-\nu_1),\]
  \[\bmu=(\mu_1\geq\mu_2\geq\ldots\geq\mu_{n-1}\geq\mu_n\geq-\mu_n\geq-\mu_{n-1}\geq\ldots\geq-\mu_1),\]
  also we allow sequences ({\em not} signatures) $\bmu$ such that
\[\bmu=(\mu_1\geq\mu_2\geq\ldots\geq\mu_{n-1}\geq-\mu_n\leq\mu_n\geq-\mu_{n-1}\geq\ldots\geq-\mu_1),\]
 where $(\mu_1\geq\mu_2\geq\ldots\geq\mu_n>0)$ is a partition. 
In the even case~\ref{OSE}(b)
  \[\bmu=(\mu_1\geq\mu_2\geq\ldots\geq\mu_{n-1}\geq0\geq-\mu_{n-1}\geq\ldots\geq-\mu_1),\]
\[\bnu=(\nu_1\geq\nu_2\geq\ldots\geq\nu_{n-1}\geq\nu_n\geq-\nu_n\geq-\nu_{n-1}\geq\ldots\geq-\nu_1),\]
  also we allow sequences ({\em not} signatures) $\bnu$ such that
\[\bnu=(\nu_1\geq\nu_2\geq\ldots\geq\nu_{n-1}\geq-\nu_n\leq\nu_n\geq-\nu_{n-1}\geq\ldots\geq-\nu_1),\]
where $(\nu_1\geq\nu_2\geq\ldots\geq\nu_n>0)$ is a partition.

  Note that in the odd case the pair $(L'_\bmu,L_\bnu)$ lies in the $\GL(N,\bF)$-orbit
  in the ambient product
  $\Gr_{\GL_N}\times\Gr_{\GL_N}\supset\Gr_{\GL_{N-1}}\times\Gr_{\GL_N}$ corresponding to a signature
  $\bta$, where
  \[\bta:=(\mu_1+\nu_1\geq\ldots\geq\mu_{n-1}+\nu_{n-1}\geq|\mu_n|+\nu_n\geq0\geq-|\mu_n|-\nu_n
  \geq\ldots\geq-\mu_1-\nu_1),\]
  and in the even case the pair $(L'_\bmu,L_\bnu)$ lies in the $\GL(N,\bF)$-orbit
  in the ambient product
  $\Gr_{\GL_N}\times\Gr_{\GL_N}\supset\Gr_{\GL_{N-1}}\times\Gr_{\GL_N}$ corresponding to a signature
  $\bta$, where
  \[\bta:=(\mu_1+\nu_1\geq\ldots\geq\mu_{n-1}+\nu_{n-1}\geq|\nu_n|\geq-|\nu_n|\geq-\mu_{n-1}-\nu_{n-1}
  \geq\ldots\geq-\mu_1-\nu_1).\]

  In all the cases listed, $L'_\bmu$ corresponds to a dominant coweight of $\SO_{N-1}$,
  while $L_\bnu$ corresponds to an {\em anti}\!\! dominant coweight of $\SO_N$. It follows that
  $\on{Stab}_{\SO(N-1,\bF)}(L'_\bmu,L_\bnu)\subset\SO(N-1,\bO)$.
  Similarly, $\on{Stab}_{\GL(N-1,\bF)}(L'_\bmu,L_\bnu)\subset\GL(N-1,\bO)$. The latter stabilizer
  has the connected unipotent radical and the reductive quotient
  $\on{Stab}^{\on{red}}_{\GL(N-1,\bF)}(L'_\bmu,L_\bnu)\simeq\prod_{i\in\BZ}\GL_{m_i}$, where
  $m_i$ is defined as follows. We consider a sequence $\boldsymbol{\alpha}$ of length $2N-1$
  obtained as a shuffle of $\bnu$ and $\bmu$, i.e.\ in the odd case
  \[\boldsymbol{\alpha}=(\nu_1,\mu_1,\nu_2,\mu_2,\ldots,\nu_n,|\mu_n|,0,-|\mu_n|,-\nu_n,\ldots,
  -\mu_1,-\nu_1),\]
  while in the even case
  \[\boldsymbol{\alpha}=(\nu_1,\mu_1,\ldots,\nu_{n-1},\mu_{n-1},|\nu_n|,0,-|\nu_n|,-\mu_{n-1},-\nu_{n-1},
  \ldots,-\mu_1,-\nu_1).\]
  Now we consider a signature $\boldsymbol{\beta}$ of length $2N-2$ formed by the sums of two
  consecutive terms of $\boldsymbol{\alpha}$:
  \[(\beta_1=\nu_1+\mu_1,\ \beta_2=\mu_1+\nu_2,\ \beta_3=\nu_2+\mu_2,\ \ldots,\
  \beta_{2N-2}=-\mu_1-\nu_1).\]
  Let $n_i$ be the multiplicity of an integer $i$ in the sequence $\boldsymbol{\beta}$.
  Finally, $m_i:=\lfloor n_i/2\rfloor$.
  
  We see in particular that
  $\on{Stab}^{\on{red}}_{\GL(N-1,\bF)}(L'_\bmu,L_\bnu)$ and $\on{Stab}_{\GL(N-1,\bF)}(L'_\bmu,L_\bnu)$
  are both connected. Viewing $\SO_M$ as the connected component of an involution of $\GL_M$,
  we see that $\on{Stab}_{\SO(N-1,\bF)}(L'_\bmu,L_\bnu)$ has the connected unipotent radical and
  the reductive quotient
  $\on{Stab}^{\on{red}}_{\SO(N-1,\bF)}(L'_\bmu,L_\bnu)\simeq\SO_{m_0}\times\prod_{i>0}\GL_{m_i}$ that
  is connected as well.
\end{proof}

\subsection{Deequivariantized Ext algebra}
\label{deeq}
In the blinking notation of~\S\ref{setup} let $\IC_{\lambda_\sfs}$ (resp.\ $\IC_{\lambda_\sfb}$)
stand for the IC-sheaf of the orbit closure
$\IC(\ol\Gr{}_{\SO_{N-1}}^{\lambda_\sfs})$ (resp.\ $\IC(\ol\Gr{}_{\SO_N}^{\lambda_\sfb})$). Then the
convolution $\IC_{\lambda_\sfs}*\IC_{\lambda_\sfb}=\IC^{\lambda_\sfs}_0\star\IC^0_{\lambda_\sfb}$ (the fusion)
is the direct sum of $\IC^{\lambda_\sfs}_{\lambda_\sfb}$ and some sheaves with support in the boundary
of $\BO^{\lambda_\sfs}_{\lambda_\sfb}$. Actually we will see in~Corollary~\ref{small} below that
$\IC_{\lambda_\sfs}*\IC_{\lambda_\sfb}=\IC^{\lambda_\sfs}_0\star\IC^0_{\lambda_\sfb}=\IC^{\lambda_\sfs}_{\lambda_\sfb}$.

We restrict the left action of $D^b_{\SO(N-1,\bO)}(\Gr_{\SO_{N-1}})$
(resp.\ the right action of $D^b_{\SO(N,\bO)}(\Gr_{\SO_N})$) on $D^b_{\SO(N-1,\bO)}(\Gr_{\SO_N})$ to
the left action of $\Perv_{\SO(N-1,\bO)}(\Gr_{\SO_{N-1}})\cong\Rep(\SO_{N-1}^\vee)$
(resp.\ to the right action of $\Perv_{\SO(N,\bO)}(\Gr_{\SO_N})\cong\Rep(\SO_N^\vee)$).
Let $D_{\SO(N-1,\bO)}^\deeq(\Gr_{\SO_N})$ denote the corresponding deequivariantized category
(see~\cite{ag} in the setting of abelian categories and~\cite{1-aff} in the setting
of dg-categories).
We have

\begin{multline}
  \label{def deeq a}
  \RHom_{D^\deeq_{\SO(N-1,\bO)}(\Gr_{\SO_N})}(\CF,\CG)\\
  =\bigoplus_{\lambda_\sfs\in\Lambda_\sfs^+,\ \lambda_\sfb\in\Lambda_\sfb^+}\RHom_{D^b_{\SO(N-1,\bO)}(\Gr_{\SO_N})}
  (\CF,\IC_{\lambda_\sfs}*\CG*\IC_{\lambda_\sfb})\otimes V_{\lambda_\sfs}^*\otimes V_{\lambda_\sfb}^*
\end{multline}
(recall that the geometric Satake equivalence takes $\IC_{\lambda_\sfs}$ to $V_{\lambda_\sfs}$,
and $\IC_{\lambda_\sfb}$ to $V_{\lambda_\sfb}$, notations of~\S\ref{setup}).

\begin{lem}
  \label{formality}
  The dg-algebra $\RHom_{D^\deeq_{\SO(N-1,\bO)}(\Gr_{\SO_N})}(E_0,E_0)$ is formal, i.e.\ it is
  quasiisomorphic to the graded algebra $\Ext_{D^\deeq_{\SO(N-1,\bO)}(\Gr_{\SO_N})}^\bullet(E_0,E_0)$ with trivial
  differential.
\end{lem}

\begin{proof}
  The argument essentially repeats the one in the proof of~\cite[Lemma 3.9.1]{bfgt}.
  The desired result follows from the purity of
  $\Ext_{D^b_{\SO(N-1,\bO)}(\Gr_{\SO_N})}(E_0,\IC^{\lambda_\sfs}_{\lambda_\sfb})$.
  We know that $\IC^{\lambda_\sfs}_{\lambda_\sfb}$ is a direct summand in
  $\IC_{\lambda_\sfs}*\IC_{\lambda_\sfb}$, and it suffices
  to prove the purity of $i_0^!(\IC_{\lambda_\sfs}*\IC_{\lambda_\sfb})$ where $i_0$ stands for the closed
  embedding of the base point $0$ into $\Gr_{\SO_N}$.

Assume first that $N\geq4$.
  Let $\varpi_{N-1}$ (resp.\ $\varpi_N$) denote the minuscule fundamental coweight of $\SO_{N-1}$
  (resp.\ of $\SO_N$). The corresponding closed $\SO(N-1,\bO)$-orbit
  $\Gr^{\varpi_{N-1}}_{\SO_{N-1}}\subset\Gr_{\SO_{N-1}}$ (resp.\ closed $\SO(N,\bO)$-orbit
  $\Gr^{\varpi_N}_{\SO_N}\subset\Gr_{\SO_N}$) is isomorphic to a smooth $(N-3)$-dimensional quadric
  $Q^{N-3}$ (resp.\ to a smooth $(N-2)$-dimensional quadric $Q^{N-2}$). It is well known that
  for any $\lambda_\sfs\in\Lambda_\sfs^+$ (resp.\ $\lambda_\sfb\in\Lambda_\sfb^+$),
  $\IC_{\lambda_\sfs}$ is a
  direct summand in a suitable convolution power $\IC_{\varpi_{N-1}}*\cdots*\IC_{\varpi_{N-1}}$
  (resp.\ $\IC_{\lambda_\sfb}$ is a direct summand in $\IC_{\varpi_N}*\cdots*\IC_{\varpi_N}$)
  (equivalently, by the geometric Satake equivalence, the defining representation of a
  symplectic group (resp.\ of a special orthogonal group) generates its representations' category
  with respect to tensor products and direct summands~\cite{w}).
  Thus it suffices to prove the purity of
  \[i_0^!(\IC_{\varpi_{N-1}}*\cdots*\IC_{\varpi_{N-1}}*\IC_{\varpi_N}*\cdots*\IC_{\varpi_N}).\]
  The latter convolution is the direct image of the constant sheaf on the smooth convolution
  diagram $\Gr^{\varpi_{N-1}}_{\SO_{N-1}}\wtimes\cdots\wtimes\Gr^{\varpi_{N-1}}_{\SO_{N-1}}
  \wtimes\Gr^{\varpi_N}_{\SO_N}\wtimes\cdots\wtimes\Gr^{\varpi_N}_{\SO_N}
  \stackrel{\bm}{\longrightarrow}\Gr_{\SO_N}$. Hence it suffices to check that the fiber $\bm^{-1}(0)$
  over the base point is a union of cells. Now under the action of the loop rotation $\BG_m$,
  every point in an open neighbourhood of $0\in\Gr_{\SO_N}$ flows away from $0$. It follows that
  $\bm^{-1}(0)$ coincides with the $\BG_m$-attractor to the union $F_0$ of the $\BG_m$-fixed
  point components in the above convolution diagram lying over $0\in\Gr_{\SO_N}$.
  By the classical Bialynicki-Birula argument, this attractor is a union of cells if $F_0$
  itself is a union of cells. Finally, a Cartan subgroup of $\SO_{N-1}$ has finitely many
  fixed points in the above convolution diagram, and the same Bialynicki-Birula argument implies
  that $F_0$ is a union of cells.

  The proof for $N=3$ is essentially the same.  The only difference is that the standard
  (2-dimensional) representation of $G_\sfs=\SO_2\cong\BG_m$ corresponds to $\IC(Q^0)$ which
  is the sum of two skyscrapers of the two points of the ``$0$-dimensional quadric'' $Q^0$.
  After replacing $\IC_{\varpi_{N-1}}$ with $\IC(Q^0)$, the same argument goes through. Alternatively,
  since $\SO_2\cong\GL_1$, and $\SO_3\cong\on{PGL}_2$, our lemma in case $N=3$ directly follows
  from~\cite[Lemma 3.9.1]{bfgt}.
\end{proof}

We denote the dg-algebra
$\Ext^\bullet_{D^\deeq_{\SO(N-1,\bO)}(\Gr_{\SO_N})}(E_0,E_0)$ (with trivial differential)
by $\fE^\bullet$. Since it is an Ext-algebra in the deequivariantized category between objects
induced from the original category, it is automatically
equipped with an action of $\SO(V_0)\times\Sp(V_1)=\sG_{\bar0}$ (notations of~\S\ref{osp}),
and we can consider the corresponding triangulated category $D^{\sG_{\bar0}}_\perf(\fE^\bullet)$.

\begin{lem}
  \label{purity}
There is a canonical equivalence $D^{\sG_{\bar0}}_\perf(\fE^\bullet)\iso D^b_{\SO(N-1,\bO)}(\Gr_{\SO_N})$.
\end{lem}

\begin{proof}
  Same as the one of~\cite[Lemma 3.9.2]{bfgt}.
\end{proof}

\subsection{Equivariant cohomology}
\label{equi coh}
The affine Grassmannian $\Gr_{\SO_N}$ has two connected components $\Gr_{\SO_N}^\odd$ and
$\Gr_{\SO_N}^\even$ (recall that $N>2$).
 In the blinking notation of~\S\ref{setup}, the equivariant cohomology ring
$H^\bullet_{\SO(N,\bO)}(\Gr_{\SO_N}^\odd)=H^\bullet_{\SO(N,\bO)}(\Gr_{\SO_N}^\even)\cong\BC[T\Sigma_\sfb]$.
 This is a theorem of V.~Ginzburg~\cite{g1} (for a published
account see e.g.~\cite[Theorem 1]{bf}). It follows that
\[H^\bullet_{\SO(N-1,\bO)}(\Gr_{\SO_N}^\odd)=
H^\bullet_{\SO(N-1,\bO)}(\Gr_{\SO_N}^\even)\cong\BC[\Sigma_\sfs\times_{\Sigma_\sfb}T\Sigma_\sfb]\]
(with respect to the morphism $\varPi_{\sfs\sfb}\colon\Sigma_\sfs\to\Sigma_\sfb$,
notations of~\S\ref{setup}).

\begin{lem}
  \label{inj}
  For any $\lambda_\sfs\in\Lambda_\sfs^+,\ \lambda_\sfb\in\Lambda_\sfb^+$, the natural morphism
  \begin{multline*}\Ext_{D^b_{\SO(N-1,\bO)}(\Gr_{\SO_N})}(E_0,\IC^{\lambda_\sfs}_{\lambda_\sfb})\\
  \to\Hom_{H^\bullet_{\SO(N-1,\bO)}(\Gr_{\SO_N})}\left(H^\bullet_{\SO(N-1,\bO)}(\Gr_{\SO_N},E_0),
  H^\bullet_{\SO(N-1,\bO)}(\Gr_{\SO_N},\IC^{\lambda_\sfs}_{\lambda_\sfb})\right)
  \end{multline*}
  is injective.
  \end{lem}

\begin{proof}
  It suffices to prove that the natural morphism
  \begin{multline*}
  \Ext_{D^b_{\SO(N-1,\bO)}(\Gr_{\SO_N})}(E_0,\IC^{\lambda_\sfs}_{\lambda_\sfb})\\
  \to\Hom_{H^\bullet_{\SO(N-1,\bO)}(\pt)}\left(H^\bullet_{\SO(N-1,\bO)}(\Gr_{\SO_N},E_0),
  H^\bullet_{\SO(N-1,\bO)}(\Gr_{\SO_N},\IC^{\lambda_\sfs}_{\lambda_\sfb})\right)
  \end{multline*}
  (in the RHS we take Hom over the equivariant cohomology of the point) is injective.
  As in the proof of~Lemma~\ref{formality}, it suffices to check the injectivity for
  the iterated convolution $\IC_{\varpi_{N-1}}*\cdots*\IC_{\varpi_{N-1}}*\IC_{\varpi_N}*\cdots*\IC_{\varpi_N}$
  in place of $\IC^{\lambda_\sfs}_{\lambda_\sfb}$. Due to purity established in {\em loc.\ \!cit.}\ (= the
  proof of~Lemma~\ref{formality}),
  the LHS is a free $H^\bullet_{\SO(N-1,\bO)}(\pt)$-module with the space of generators isomorphic to
  the costalk of the above convolution at the base point $0\in\Gr_{\SO_N}$, that is to
  $H^\bullet(\bm^{-1}(0))$ (notations of {\em loc.\ cit.}).
  The RHS is also a free
  $H^\bullet_{\SO(N-1,\bO)}(\pt)$-module with the space of generators isomorphic to
  $H^\bullet\left(\Gr^{\varpi_{N-1}}_{\SO_{N-1}}\wtimes\cdots\wtimes\Gr^{\varpi_{N-1}}_{\SO_{N-1}}
  \wtimes\Gr^{\varpi_N}_{\SO_N}\wtimes\cdots\wtimes\Gr^{\varpi_N}_{\SO_N}\right)$.
  It contains $H^\bullet(\bm^{-1}(0))$ as a direct summand since the convolution diagram has a
  cellular decomposition compatible with the one for $\bm^{-1}(0)$, see {\em loc.\ cit.}
\end{proof}

\subsection{Calculation of the Ext algebra}
\label{calcul ext}
Recall that \[\BC[T\Sigma_\sfs]\cong H^\bullet_{\SO(N-1,\bO)}(\Gr_{\SO_{N-1}}^\odd)
\cong H^\bullet_{\SO(N-1,\bO)}(\Gr_{\SO_{N-1}}^\even),\] and
\[\BC[T\Sigma_\sfb]\cong H^\bullet_{\SO(N,\bO)}(\Gr_{\SO_N}^\odd)
\cong H^\bullet_{\SO(N,\bO)}(\Gr_{\SO_N}^\even).\] Moreover, for
$\lambda_\sfs\in\Lambda^+_\sfs$ (resp.\ $\lambda_\sfb\in\Lambda^+_\sfb$) we have canonical isomorphisms
of $\BC[T\Sigma_\sfs]$-modules (resp.\ $\BC[T\Sigma_\sfb]$-modules)
$\kappa_\sfs(V_{\lambda_\sfs})\cong H^\bullet_{\SO(N-1,\bO)}(\Gr_{\SO_{N-1}},\IC_{\lambda_\sfs})$
(resp.\ $\kappa_\sfb(V_{\lambda_\sfb})\cong H^\bullet_{\SO(N,\bO)}(\Gr_{\SO_N},\IC_{\lambda_\sfb})$)
(for Kostant functors $\kappa$ see~\S\ref{setup}).
This is also a theorem of V.~Ginzburg~\cite{g1} (for a published account see
e.g.~\cite[Theorem 6 and Lemma 9]{bf}). It follows that we have a canonical isomorphism
of $\BC[\Sigma_\sfs\times_{\Sigma_\sfb}T\Sigma_\sfb]$-modules
$d\varPi_{\sfs\sfb}^*\kappa_\sfb(V_{\lambda_\sfb})\cong H^\bullet_{\SO(N-1,\bO)}(\Gr_{\SO_N},\IC_{\lambda_\sfb})$.

\begin{lem}
  \label{tensor}
  For $\lambda_\sfs\in\Lambda^+_\sfs,\ \lambda_\sfb\in\Lambda^+_\sfb$ we have a canonical
  isomorphism of $\BC[\Sigma_\sfs\times_{\Sigma_\sfb}T\Sigma_\sfb]$-modules
  \[\kappa_\sfs(V_{\lambda_\sfs})\otimes_{\BC[\Sigma_\sfs]}
  d\varPi_{\sfs\sfb}^*\kappa_\sfb(V_{\lambda_\sfb})\big)\iso
  H^\bullet_{\SO(N-1,\bO)}(\Gr_{\SO_N},\IC_{\lambda_\sfs}*\IC_{\lambda_\sfb}).\]
\end{lem}

\begin{proof}
  By the classical argument going back to Drinfeld,
  $\IC_{\lambda_\sfs}*\IC_{\lambda_\sfb}\cong\IC_{\lambda_\sfs}\star\IC_{\lambda_\sfb}$, where the fusion $\star$
  is defined by taking nearby cycles in the Beilinson-Drinfeld Grassmannian
  $\Gr_{\on{BD}}\stackrel{\pi}{\longrightarrow}\BA^1$. The fiber $\pi^{-1}(0)$ is $\Gr_{\SO_N}$,
  and for $x\ne0$, the fiber $\pi^{-1}(x)$ is $\Gr_{\SO_{N-1}}\times\Gr_{\SO_N}$. We have a tautological
  closed embedding $\Gr_{\on{BD}}\hookrightarrow\Gr_{\SO_N,\on{BD}}$ into the usual Beilinson-Drinfeld
  Grassmannian of $\SO_N$. The cospecialization morphism to the cohomology of a nearby fiber
  \begin{multline*}H^\bullet_{\SO(N-1,\bO)}(\Gr_{\SO_N},\IC_{\lambda_\sfs}\star\IC_{\lambda_\sfb})=
  H^\bullet_{\SO_{N-1}}(\Gr_{\SO_N},\IC_{\lambda_\sfs}\star\IC_{\lambda_\sfb})\\
  \to H^\bullet_{\SO_{N-1}}(\Gr_{\SO_{N-1}}\times\Gr_{\SO_N},\IC_{\lambda_\sfs}\boxtimes\IC_{\lambda_\sfb})
  \end{multline*}
  is an isomorphism (due to properness), and
  is compatible with the cospecialization morphism of the cohomology of ambient spaces
  $H^\bullet_{\SO_{N-1}}(\Gr_{\SO_N})\to H^\bullet_{\SO_{N-1}}(\Gr_{\SO_{N-1}}\times\Gr_{\SO_N})$, and the diagram
  formed by the cospecialization morphisms and restriction with respect to the above closed
  embedding of Beilinson-Drinfeld Grassmannians commutes:
  \begin{equation*}
    \begin{CD}
      H^\bullet_{\SO_{N-1}}(\Gr_{\SO_N}) @>>>  H^\bullet_{\SO_{N-1}}(\Gr_{\SO_{N-1}}\times\Gr_{\SO_N})\\
      @AAA @AAA\\
      H^\bullet_{\SO_N}(\Gr_{\SO_N}) @>>> H^\bullet_{\SO_N}(\Gr_{\SO_N}\times\Gr_{\SO_N}).
    \end{CD}
  \end{equation*}
  Finally, the following diagram commutes as well:
  \begin{equation*}
    \begin{CD}
      \BC[T\Sigma_\sfb] @>{\on{add}^*}>> \BC[T\Sigma_\sfb\times_{\Sigma_\sfb}T\Sigma_\sfb]\\
      @VV{\wr}V @VV{\wr}V\\
      H^\bullet_{\SO_N}(\Gr_{\SO_N}) @>>> H^\bullet_{\SO_N}(\Gr_{\SO_N}\times\Gr_{\SO_N}),
    \end{CD}
  \end{equation*}
  where $\on{add}\colon T\Sigma_\sfb\times_{\Sigma_\sfb}T\Sigma_\sfb\to T\Sigma_\sfb$ stands
  for the fiberwise addition morphism.
  The lemma follows.
\end{proof}

Now recall the minuscule closed orbits
$Q^{N-3}\cong\Gr^{\varpi_{N-1}}_{\SO_{N-1}}\subset\Gr^{\varpi_N}_{\SO_N}\cong Q^{N-2}$ (smooth quadrics).
For $N>3$ we have
\begin{multline*}
  \Ext_{D^b_{\SO(N-1,\bO)}(\Gr_{\SO_N})}^\bullet(E_0,\IC_{\varpi_{N-1}}*E_0*\IC_{\varpi_N})\\
=\Ext_{D^b_{\SO(N-1,\bO)}(\Gr_{\SO_N})}^\bullet(\IC_{\varpi_{N-1}}*E_0,E_0*\IC_{\varpi_N})\\
=\Ext_{D^b_{\SO(N-1,\bO)}(\Gr_{\SO_N})}^\bullet(\IC(Q^{N-3}),\IC(Q^{N-2})).
\end{multline*}
(In case $N=3$ we replace $\IC_{\varpi_{N-1}}$ with $\IC(Q^0)$ the same way as in the last
paragraph of the proof of Lemma~\ref{formality}.)
Since $Q^{N-3}\subset Q^{N-2}$ is a smooth divisor, we have a canonical element
\[h\in\Ext_{D^b_{\SO(N-1,\bO)}(\Gr_{\SO_N})}^1(\IC(Q^{N-3}),\IC(Q^{N-2})).\] Hence we obtain the subspace
\[h\otimes V_0^*\otimes V_1^*\cong h\otimes V_0\otimes V_1\subset\fE^1:=
\Ext_{D^\deeq_{\SO(N-1,\bO)}(\Gr_{\SO_N})}^1(E_0,E_0),\] cf.~(\ref{def deeq a}).
We will denote this subspace simply by $V_0\otimes V_1$. Thus we obtain a homomorphism from
the free tensor algebra \[\phi^\bullet\colon T(\Pi(V_0\otimes V_1)[-1])\to\fE^\bullet:=
\Ext_{D^\deeq_{\SO(N-1,\bO)}(\Gr_{\SO_N})}^\bullet(E_0,E_0).\]

\begin{lem}
  \label{phiso}
  The homomorphism $\phi^\bullet$ factors through the projection
  \[T(\Pi(V_0\otimes V_1)[-1])\twoheadrightarrow\Sym(\Pi(V_0\otimes V_1)[-1])=\fG^\bullet,\]
  and induces an isomorphism $\fG^\bullet\iso\fE^\bullet$.
\end{lem}

\begin{proof}
  We have a tautological isomorphism
  \begin{multline*}\fG^\bullet\cong\Ext^\bullet_{D^{\sG_{\bar0},\deeq}_\perf(\fG^\bullet)}(\fG^\bullet,\fG^\bullet)\\
  =\bigoplus_{\lambda_\sfs\in\Lambda_\sfs^+,\ \lambda_\sfb\in\Lambda_\sfb^+}
  \Ext^\bullet_{D^{\sG_{\bar0}}_\perf(\fG^\bullet)}(\fG^\bullet,V_{\lambda_\sfs}\otimes\fG^\bullet\otimes
  V_{\lambda_\sfb})\otimes V_{\lambda_\sfs}^*\otimes V_{\lambda_\sfb}^*.
  \end{multline*}
  By Proposition~\ref{kostant} below, the Kostant functors induce an isomorphism
  \begin{multline*}\fG^\bullet\cong\bigoplus_{\lambda_\sfs\in\Lambda_\sfs^+,\ \lambda_\sfb\in\Lambda_\sfb^+}
    \Ext^\bullet_{D^{\sG_{\bar0}}_\perf(\fG^\bullet)}(\fG^\bullet,V_{\lambda_\sfs}\otimes\fG^\bullet\otimes
    V_{\lambda_\sfb})\otimes V_{\lambda_\sfs}^*\otimes V_{\lambda_\sfb}^*\\
    \iso\bigoplus_{\lambda_\sfs\in\Lambda_\sfs^+,\ \lambda_\sfb\in\Lambda_\sfb^+}
    \Hom_{\BC[T\Sigma_\sfs]}\big(\BC[\Sigma_\sfs],
  \kappa_\sfs(V_{\lambda_\sfs})\otimes_{\BC[\Sigma_\sfs]}d\varPi_{\sfs\sfb}^*\kappa_\sfb(V_{\lambda_\sfb})\big)
  \otimes V_{\lambda_\sfs}^*\otimes V_{\lambda_\sfb}^*.
  \end{multline*}
  Here we view $\Sigma_\sfs$ as the zero section of the tangent bundle $T\Sigma_\sfs$, so that
  $\BC[\Sigma_\sfs]$ acquires a structure of $\BC[T\Sigma_\sfs]$-module.
  Comparing with~Lemma~\ref{tensor}, by~Lemma~\ref{inj} we obtain an injective homomorphism from
  the topological Ext-algebra
  to the algebraic one: $\fE^\bullet\hookrightarrow\fG^\bullet$. Since $\fG^\bullet$ is commutative,
  we conclude that $\fE^\bullet$ is commutative as well, i.e.\ $\phi^\bullet$ does factor through
  $\bar\phi{}^\bullet\colon\Sym(\Pi(V_0\otimes V_1)[-1])=\fG^\bullet\to\fE^\bullet$.
  Finally, since the composition
  $\fG^\bullet\stackrel{\bar\phi{}^\bullet}{\longrightarrow}\fE^\bullet\hookrightarrow\fG^\bullet$
  is identity on the generators $\Pi(V_0\otimes V_1)$ of $\fG^\bullet$, we conclude that
  $\bar\phi{}^\bullet$ is an isomorphism.
\end{proof}

Now the existence of the desired equivalence $\Phi$ of~Theorem~\ref{main}(a) follows
from~Lemma~\ref{purity} and~Lemma~\ref{phiso}. Furthermore, the claims of~Theorem~\ref{main}(b,e)
are proved exactly as~\cite[Corollary 3.8.1(a,c)]{bfgt}.

\begin{cor}
  \label{small}
  We have $\IC_{\lambda_\sfs}*\IC_{\lambda_\sfb}
  =\IC^{\lambda_\sfs}_0\star\IC^0_{\lambda_\sfb}=\IC^{\lambda_\sfs}_{\lambda_\sfb}$.
\end{cor}

\begin{proof}
  By construction, $\Phi(V_{\lambda_\sfs}\otimes\fG^\bullet\otimes V_{\lambda_\sfb})
  =\IC_{\lambda_\sfs}*\IC_{\lambda_\sfb}$. But $V_{\lambda_\sfs}\otimes\fG^\bullet\otimes V_{\lambda_\sfb}$ is
  an indecomposable object of $D^{\sG_{\bar0}}_\perf(\fG^\bullet)$, hence
  $\IC_{\lambda_\sfs}*\IC_{\lambda_\sfb}=\IC^{\lambda_\sfs}_0\star\IC^0_{\lambda_\sfb}$ must be indecomposable
  as well, i.e.\ it must coincide with $\IC^{\lambda_\sfs}_{\lambda_\sfb}$.
\end{proof}

\subsection{Compatibility with the spherical Hecke actions}
\label{compa}
To finish the proof of~Theorem~\ref{main}(a) it remains to check the compatibility with the
left and right convolution actions of the monoidal spherical Hecke categories. We check the
compatibility for the left action; the verification for the right action is similar.
Our argument is similar to the one in the proof of~\cite[Lemma 3.11.1]{bfgt}.
Namely, we already know the compatibility with the convolution action of the semisimple abelian
category $\Perv_{\SO(N-1,\bO)}(\Gr_{\SO_{N-1}})$. Hence we obtain a homomorphism from the deequivariantized
Ext-algebra of the unit object of $D^b_{\SO(N-1,\bO)}(\Gr_{\SO_{N-1}})$ to the deequivariantized
Ext-algebra of the unit object of $D^b_{\SO(N-1,\bO)}(\Gr_{\SO_N})$. The corresponding RHom-algebras
are formal, and by~\cite[Proposition 7, Remarks 2,3]{bf}
it suffices to check that the above homomorphism of graded commutative algebras
coincides with $\bq_\sfs^*$. We proceed to check the desired equality on generators.

Recall the element
$h\in\Ext_{D^b_{\SO(N-1,\bO)}(\Gr_{\SO_N})}^1(\IC(Q^{N-3}),\IC(Q^{N-2}))$ of~\S\ref{calcul ext}.
Dually, we have a canonical element $h^*\in\Ext_{D^b_{\SO(N-1,\bO)}(\Gr_{\SO_N})}^1(\IC(Q^{N-2}),\IC(Q^{N-3}))$.
The composition $h^*\circ h\in\Ext_{D^b_{\SO(N-1,\bO)}(\Gr_{\SO_N})}^2(\IC(Q^{N-3}),\IC(Q^{N-3}))$ is the
multiplication by the first Chern class of the normal line bundle $\CN_{Q^{N-3}/Q^{N-2}}$.
This normal bundle is isomorphic to the line bundle $\CO(1)$ restricted from $\BP^{N-1}$ under the
tautological embedding $Q^{N-3}\subset Q^{N-2}\subset\BP^{N-1}$.

If $N\ne4$, the Picard group of the connected component $\Gr^\odd_{\SO_N}$ containing $Q^{N-3}$
is isomorphic to $\BZ$. Its ample generator is denoted $\CL_N$, the determinant line bundle.
The restriction $\CL_N|_{Q^{N-3}}$ is also isomorphic to $\CO(1)\simeq\CN_{Q^{N-3}/Q^{N-2}}$. We conclude
that $h^*\circ h=c_1(\CL_N)$ (when $N\ne4$). On the other hand, in the equivariant derived Satake
category $D^b_{\SO(N-1,\bO)}(\Gr_{\SO_{N-1}})\cong D^{G_\sfs}_\perf\big(\Sym(\fg_\sfs[-2])\big)$ the
first Chern class
\[c_1(\CL_{N-1})\in\Ext^2_{D^b_{\SO(N-1,\bO)}(\Gr_{\SO_{N-1}})}(\IC_{\varpi_{N-1}},\IC_{\varpi_{N-1}})\subset
\fg_\sfs\otimes\End(V_\sfs)\] corresponds to the canonical `action' element
$\fg_\sfs^*\cong\fg_\sfs\hookrightarrow\End(V_\sfs)$. This completes the verification of the desired
compatibility with the left action in case $N\ne4$. The case $N=4$ is left as an exercise to the
interested reader.

Theorem~\ref{main}(a) is proved.

\subsection{Some Invariant Theory}
\label{IT}
Recall the blinking notation of~\S\ref{setup}.

\begin{lem}
  \label{weyl}
  \textup{(a)} 
  The morphism $\bq_\sfs$ induces an
  isomorphism of categorical quotients
  \[(V_\sfs\otimes V_\sfb)\dslash(G_\sfs\times G_\sfb)\iso\fg_\sfs^*\dslash G_\sfs\cong\Sigma_\sfs.\]


  \textup{(b)} The following diagram
  commutes:
  \begin{equation*}
    \begin{CD}
      (V_\sfs\otimes V_\sfb)\dslash(G_\sfs\times G_\sfb) @>{\bq_\sfb}>> \fg_\sfb^*\dslash G_\sfb\\
      @V{\bq_\sfs}V{\wr}V @|\\
      \Sigma_\sfs @>{\varPi_{\sfs\sfb}}>> \Sigma_\sfb.
    \end{CD}
  \end{equation*}

  Thus the image of the complete moment map
  \[(\bq_\sfs,\bq_\sfb)\colon (V_\sfs\otimes V_\sfb)\dslash(G_\sfs\times G_\sfb)\to
  \fg_\sfs^*\dslash G_\sfs\times\fg_\sfb^*\dslash G_\sfb\cong\Sigma_\sfs\times\Sigma_\sfb\]
  identifies $(V_\sfs\otimes V_\sfb)\dslash(G_\sfs\times G_\sfb)$ with the graph of
  $\varPi_{\sfs\sfb}$. 
\end{lem}

\begin{proof}
  (a) In the odd case, the morphism $\bq_0$ is clearly dominant, so
  \[\bq_0^*\colon\BC[\fso(V_0)^*]^{\SO(V_0)}\to\BC[V_0\otimes V_1]^{\SO(V_0)\times\Sp(V_1)}\] is injective.
  It remains to prove the surjectivity of $\bq_0^*$. It is enough to prove the surjectivity of
  $\bq_0^*\colon\BC[\fso(V_0)^*]\to\BC[V_0\otimes V_1]^{\Sp(V_1)}$.
  According to the first fundamental theorem of
  the invariant theory for $\Sp(V_1)$~\cite{w}, the algebra
  $\BC[V_0\otimes V_1]^{\Sp(V_1)}$ is generated by the quadratic expressions
  $Q_{ij},\ 1\leq i<j\leq 2n$, of the following sort. We choose an orthonormal basis
  $e_1,\ldots,e_{2n}$ in $V_0$ and denote by $p_i,\ 1\leq i\leq 2n$, the corresponding
  projections $V_0\otimes V_1\to V_1$. Then
  \[Q_{ij}(v_0\otimes v_1,v'_0\otimes v'_1):=\langle p_i(v_0\otimes v_1),p_j(v'_0\otimes v'_1)\rangle.\]
  Now $\fso(V_0)$ is formed by all the skew-symmetric matrices in the above basis.
  We denote by $E_{ij}\in\fso(V_0)^*,\ 1\leq i<j\leq 2n$, the corresponding matrix element.
  Then $\bq_0^*(E_{ij})=Q_{ij}$. This proves the desired surjectivity claim.

  The argument in the even case is entirely similar.
  Note only that according to the first fundamental theorem of
  the invariant theory for $\SO(V_0)$~\cite{w}, the algebra
  $\BC[V_0\otimes V_1]^{\SO(V_0)}$ is generated by certain quadratic expressions along with
  degree $2n$ expressions (coming from determinants). But since $\dim V_1=2n-2<2n$, these
  determinants vanish identically (so that
  $\BC[V_0\otimes V_1]^{\SO(V_0)}=\BC[V_0\otimes V_1]^{\on{O}(V_0)}$).

  (b) The ring of invariant functions on $\fso(V_0)\cong\fso(V_0)^*$ is generated by the
  coefficients of the characteristic polynomial
  $\on{Char}_D(z)\nolinebreak=\nolinebreak z^{2n}\nolinebreak+\nolinebreak\sum_{i=1}^na_i(D)z^{2n-2i},$
  $D\nolinebreak\in\nolinebreak\fso(V_0)$,
  along with the Pfaffian $\on{Pfaff}(D)$. In terms of the identification
  $\BC[\fso(V_0)]^{\SO(V_0)}\cong\BC[\ft_0]^{W_0},\ a_i$
  is the $i$-th elementary symmetric polynomial in
  $\varepsilon_1^2,\ldots,\varepsilon_n^2$ (see~\S\ref{setup}),
  and $\on{Pfaff}=\varepsilon_1\cdots\varepsilon_n$.
  The ring of invariant functions on $\fsp(V_1)\cong\fsp(V_1)^*$ is generated by the
  coefficients of the characteristic polynomial
  $\on{Char}_C(z)=z^{\dim V_1}+\sum_{i=1}^{\dim V_1/2}b_i(C)z^{\dim V_1-2i},\ C\in\fsp(V_1)$.
  In terms of the identification
  $\BC[\fsp(V_1)]^{\Sp(V_1)}\cong\BC[\ft_1]^{W_1},\ b_i$
  is the $i$-th elementary symmetric polynomial in
  $\delta_1^2,\ldots,\delta_{\dim V_1/2}^2$.
  In the odd (resp.\ even) case, for $A\in\Hom(V_0,V_1)$
  we have $\on{Char}_{A^tA}(z)=\on{Char}_{AA^t}(z)$ (resp.\ $\on{Char}_{A^tA}(z)=z^2\on{Char}_{AA^t}(z)$).
  Also, in the even case $\on{Pfaff}(A^tA)=\sqrt{\det(A^tA)}=0$. The claim (b) follows.

  This completes the proof of the lemma.
\end{proof}

We will call $A\in V_0\otimes V_1$ {\em regular} if the Lie algebra
$\mathfrak{stab}_{\fso(V_0)\oplus\fsp(V_1)}(A)$ of its stabilizer
$\on{Stab}_{\SO(V_0)\times\Sp(V_1)}(A)$ has minimal possible dimension $n$ (both in even and odd cases).
Such elements form an open subset $(V_0\otimes V_1)^\reg \subset V_0\otimes V_1$.

\begin{lem}
  \label{regul}
  \textup{(a)} For $A\in V_\sfs\otimes V_\sfb$
  the following implications hold true: \[\bq_\sfb(A)\in\fg_\sfb^{*\reg}\ \Longrightarrow
  A\in(V_\sfs\otimes V_\sfb)^\reg\ \Longrightarrow \bq_\sfs(A)\in\fg_\sfs^{*\reg}.\]


  \textup{(b)} For $A\in(V_\sfs\otimes V_\sfb)^\reg$ such that $\bq_\sfb(A)$ is regular,
  we have \[\mathfrak{stab}_{\fg_\sfs}(\bq_\sfs(A))
  \xleftarrow[\on{pr}_\sfs]{}\mathfrak{stab}_{\fg_\sfs\oplus\fg_\sfb}(A)
  \xrightarrow[\on{pr}_\sfb]{\sim}\mathfrak{stab}_{\fg_\sfb}(\bq_\sfb(A)).\]
  Thus in view of Lemma~\ref{weyl}, passing to the images in categorical quotients
   we obtain a morphism
   $\pr_\sfs\pr_\sfb^{-1}\colon\varPi_{\sfs\sfb}^*\fz_\sfb\to\fz_\sfs$ of abelian Lie algebras bundles
   over $\Sigma_\sfs$.


  \textup{(c)} In view of identifications $\fz_\sfs\cong T^*\Sigma_\sfs,\ \fz_\sfb\cong T^*\Sigma_\sfb$
  of~\S\ref{setup}, the following diagram commutes:
  \begin{equation*}
    \begin{CD}
      \varPi_{\sfs\sfb}^*\fz_\sfb @>{\pr_\sfs\pr_\sfb^{-1}}>> \fz_\sfs\\
      @| @|\\
      \varPi_{\sfs\sfb}^*T^*\Sigma_\sfb @>{d^*\varPi_{\sfs\sfb}}>> T^*\Sigma_\sfs.
    \end{CD}
    \end{equation*}
\end{lem}

\begin{proof}
  (a) The first implication follows from the classification of $G_\sfs\times G_\sfb$-orbits
  in $V_\sfs\otimes V_\sfb$, see~\cite[Theorem~6.5]{kp}
  and~\cite[Proposition~4]{gl}.\footnote{We learned the argument
    from A.~Berezhnoy, cf.~\cite[Theorems 1,2,7]{b}.}
  The second implication follows from the existence of a Weierstra\ss\ section~\cite[\S8.8]{pv},
  see e.g.~\cite[Proposition 3.1.1]{mo},
  \[(V_\sfs\otimes V_\sfb)\dslash(G_\sfs\times G_\sfb)
  =\Sigma_\sfs\hookrightarrow(V_\sfs\otimes V_\sfb)^\reg.\]


  Further, if a symplectic variety $X$ is equipped with a hamiltonian action of a Lie group $G$
  with Lie algebra $\fg$ and with a moment map $\bmu\colon X\to\fg^*$, then for a point $x\in X$,
  the cokernel of the differential $d\bmu\colon T_xX\to\fg^*$ is dual to $\mathfrak{stab}_\fg(x)$.
  For (b) we may assume that $A$ lies in the image of a Weierstra\ss\ section
  $\Sigma_\sfs\hookrightarrow(V_\sfs\otimes V_\sfb)^\reg$. Then we have an exact sequence
  \[0\to\mathfrak{stab}_{\fg_\sfs\oplus\fg_\sfb}(A)\to\fg_\sfs\oplus\fg_\sfb\to
  T_A(V_\sfs\otimes V_\sfb)\to T_A\Sigma_\sfs\to0,\]
  and (b) follows from~Lemma~\ref{weyl}(b) since the differential $d\bq_\sfs$ identifies
  $T_A\Sigma_\sfs$ with $T_{\bq_\sfs(A)}\Sigma_\sfs$.


  (c) again follows from $\mathfrak{stab}_\fg(x)^*=\on{Coker}(d\bmu)$ and the last claim
    of~Lemma~\ref{weyl}.
\end{proof}

\begin{prop}
  \label{kostant}
  Given a $G_\sfs$-module $V$ and a $G_\sfb$-module $V'$, the Kostant functors of
  restriction to Kostant slices (notation of~\S\ref{setup}) induce an isomorphism

  \begin{multline*}
    \Hom_{G_\sfs\times G_\sfb\ltimes\BC[V_\sfs\otimes V_\sfb]}(\BC[V_\sfs\otimes V_\sfb],
  V\otimes\BC[V_\sfs\otimes V_\sfb]\otimes V')\\
  \iso \Hom_{T\Sigma_\sfs}(\CO_{\Sigma_\sfs},\kappa_\sfs(V)\otimes_{\BC[\Sigma_\sfs]}
  d\varPi_{\sfs\sfb}^*\kappa_\sfb(V')).
  \end{multline*}
  Here we view $\Sigma_\sfs$ as the zero section of the tangent bundle $T\Sigma_\sfs$, so that
  $\CO_{\Sigma_\sfs}$ acquires a structure of $\CO_{T\Sigma_\sfs}$-module.
\end{prop}

\begin{proof}
  Since the codimension of the complement $(V_\sfs\otimes V_\sfb)\setminus(V_\sfs\otimes V_\sfb)^\reg$
  in $V_\sfs\otimes V_\sfb$ is at least 2, the LHS can be computed after restriction to
  $(V_\sfs\otimes V_\sfb)^\reg$, and then it coincides with the RHS by~Lemma~\ref{regul}(c).
\end{proof}

\subsection{Nilpotent support and compactness}
\label{nilpotent}
We prove Theorem~\ref{main}(f). The argument repeats the proof of~\cite[Theorem 12.5.3]{aga}.
Namely, $D^{\on{comp}}_{\SO_{N-1}}(\Gr_{\SO_N})$ is generated by
$D^b_{\SO(N-1,\bO)}(\Gr_{\SO_{N-1}})*\widetilde{E}_0*D^b_{\SO(N,\bO)}(\Gr_{\SO_N})$ by the argument of
{\em loc.\ cit.} Here $\widetilde{E}_0$ stands for the averaging $\on{Av}_{\SO(N-1,\bO),!}E_0$.
Again by {\em loc.\ cit.} $\widetilde{E}_0$ is isomorphic (up to a shift) to
$\Phi(\fG^\bullet\otimes_{\Sym(\fg_\sfs[-2])^{G_\sfs}}\BC)$
(we use the homomorphism $\bq_\sfs^*\colon \Sym(\fg_\sfs[-2])\to\fG^\bullet$).
Also,~Lemma~\ref{weyl}(a) implies $\BC[V_\sfs\otimes V_\sfb]\otimes_{\BC[\fg_\sfs]^{G_\sfs}}\BC=\BC[\CN_{\bar1}]$.
Now the desired equivalence follows by the compatibility with the spherical Hecke actions.

Recall the Weierstra\ss\ section $\Sigma_\sfs\hookrightarrow V_\sfs\otimes V_\sfb$ of the proof
of~Lemma~\ref{regul}(a). For $\CA\in D_\perf^{\sG_{\bar0}}(\fG^\bullet)$ we have a canonical
isomorphism $\Gamma(\Sigma_\sfs,\CA|_{\Sigma_\sfs})\cong H^\bullet_{\SO(N-1,\bO)}(\Gr_{\SO_N},\Phi(\CA))$.
The intersection $\Sigma_\sfs\cap\CN_{\bar1}$ is just one point (a regular nilpotent element
$A\in\Hom(V_0,V_1)$), so the nilpotent support condition implies
$\dim\Gamma(\Sigma_\sfs,\CA|_{\Sigma_\sfs})<\infty$. Conversely, since the support of $\CA$ is
invariant with respect to dilations, the condition $\dim\Gamma(\Sigma_\sfs,\CA|_{\Sigma_\sfs})<\infty$
implies $\on{supp}\CA\subset\CN_{\bar1}$.

This completes the proof of Theorem~\ref{main}(f).

\subsection{The monoidal property of $\Phi$}
\label{monoidal}
The argument is similar to that of~\cite[\S3.16]{bfgt}.
The monoidal structure $\otimes_{\fG^\bullet}$ on $D^{\sG_{\bar0}}_\perf(\fG^\bullet)$ is defined via
the kernel $\BC[\Delta]^\bullet$: the diagonal $\sG_{\bar0}$-equivariant dg-$\fG^\bullet$-trimodule.
The fusion monoidal structure $\star$ on $D^b_{\SO_{N-1}}(\Gr_{\SO_N})$ transferred to
$D^{\sG_{\bar0}}_\perf(\fG^\bullet)$ via the equivalence $\Phi$ is also defined via a kernel $\sK^\bullet$
(a $\sG_{\bar0}$-equivariant dg-$\fG^\bullet$-trimodule). We have to construct an isomorphism of
$\sG_{\bar0}$-equivariant dg-$\fG^\bullet$-trimodules $\BC[\Delta]^\bullet\iso\sK^\bullet$.

The purity of $\star$ implies the formality of $\sK^\bullet$, and it suffices to identify
$\BC[\Delta]^\bullet\iso\sK^\bullet$ as trimodules over the commutative graded algebra $\fG^\bullet$.
We know that the deequivariantized category $D^\deeq_{\SO_{N-1}}(\Gr_{\SO_N})$ is generated by $E_0$.
Furthermore, in the induced monoidal structure $\star$ of $D^\deeq_{\SO_{N-1}}(\Gr_{\SO_N})$ we have
$E_0\star E_0=E_0$, and finally, $\Ext_{D^\deeq_{\SO_{N-1}}(\Gr_{\SO_N})}(E_0,E_0)=\fG^\bullet$. The desired
isomorphism follows.

This completes the proof of the monoidal property of $\Phi$ along with~Theorem~\ref{main}.

\section{Complements}
\label{3}

\subsection{Loop rotation and quantization}
\label{loop}
We have $H_{\BG_m}^\bullet(\pt)=\BC[\hbar]$. We consider the ``graded Weyl algebra'' $\fD^\bullet$
of $V_0\otimes V_1$: a $\BC[\hbar]$-algebra generated by $V_0\otimes V_1$ with relations
$[v_0\otimes v_1,v'_0\otimes v'_1]=(v_0,v'_0)\cdot\langle v_1,v'_1\rangle\cdot\hbar$ (notation
of~\S\ref{setup}). It is equipped with the grading $\deg(v_0\otimes v_1)=1,\ \deg\hbar=2$.
We will view it as a dg-algebra with trivial differential, equipped with a natural action
of $\sG_{\bar0}=\SO(V_0)\times\Sp(V_1)$.

\begin{thm}
  \label{quantum}
  There exists an equivalence of triangulated categories
  $\Phi_\hbar\colon D^{\sG_{\bar0}}_\perf(\fD^\bullet)\iso D^b_{\SO(N-1,\bO)\rtimes\BG_m}(\Gr_{\SO_N})$
  commuting with the actions of the monoidal spherical Hecke categories
  $\Perv_{\SO(N-1,\bO)\rtimes\BG_m}(\Gr_{\SO_{N-1}})$ and
  $\Perv_{\SO(N,\bO)\rtimes\BG_m}(\Gr_{\SO_N})$ by the left and right convolutions.
\end{thm}

\begin{proof}
  We essentially repeat the argument of~\cite[\S5.2]{bfgt}. We set
  $\fE_\hbar^\bullet:=\Ext^\bullet_{D_{\SO(N-1,\bO)\rtimes\BG_m}^\deeq(\Gr_{\SO_N})}(E_0,E_0)$. Since it is an
  Ext-algebra in the deequivariantized category, it is automatically equipped with an action of
  $\SO(V_0)\times\Sp(V_1)=\sG_{\bar0}$, and we can consider the corresponding triangulated category
  $D^{\sG_{\bar0}}(\fE_\hbar^\bullet)$. Similarly to~Lemma~\ref{purity}, there is a canonical
  equivalence $D^{\sG_{\bar0}}(\fE_\hbar^\bullet)\iso D^b_{\SO(N-1,\bO)\rtimes\BG_m}(\Gr_{\SO_N})$.
  It remains to construct an isomorphism $\phi_\hbar^\bullet\colon\fD^\bullet\iso\fE_\hbar^\bullet$.

  Note that $\fE_\hbar^\bullet$ is a $\BC[\hbar]$-algebra, and
  \[\fE_\hbar^\bullet/(\hbar=0)=\fE^\bullet\cong\fG^\bullet=\Sym(\Pi(V_0\otimes V_1)[-1]),\] so that
  $\fE^\bullet$ acquires a Poisson bracket from this deformation. We claim that this Poisson bracket
  arises from the symplectic form $(,)\otimes\langle,\rangle$ on $V_0\otimes V_1$.
  Indeed, by construction, this Poisson bracket is $\SO(V_0)\times\Sp(V_1)$-invariant of degree
  $-1$. There is a unique such bracket up to a multiplicative constant, and we just have to
  determine this constant. We may and will forget the grading. The desired constant is determined
  by the condition that the moment map
  \[\bq_0^*\colon\BC[\fso(V_0)^*]\to\BC[V_0\otimes V_1]\cong\Sym(V_0\otimes V_1)\cong\fE\] is Poisson
  (where $\fE$ stands for $\fE^\bullet$ with grading forgotten). The verification of this
  condition is identical in the odd and even cases, and we consider the odd case only.
  We have functors
  \[\Upsilon_\hbar\colon D^b_{\SO(N-1,\bO)\rtimes\BG_m}(\Gr_{\SO_{N-1}})\to
    D^b_{\SO(N-1,\bO)\rtimes\BG_m}(\Gr_{\SO_N}),\ \CF\mapsto\CF*E_0;\]
  \[\Upsilon\colon D^b_{\SO(N-1,\bO)}(\Gr_{\SO_{N-1}})\to
    D^b_{\SO(N-1,\bO)}(\Gr_{\SO_N}),\ \CF\mapsto\CF*E_0.\]
  By the argument of~\S\ref{compa}, the diagram
  \begin{equation}
    \begin{CD}
      D^{\SO_{N-1}}_\perf\big(\Sym(\fso(V_0)[-2])\big) @>{\bq_0^*}>>
      D^{\sG_{\bar0}}_\perf\big(\Sym(\Pi(V_0\otimes V_1)[-1])\big)\\
      @V{\boldsymbol\beta}V{\wr}V @V{\Phi}V{\wr}V\\
      D^b_{\SO(N-1,\bO)}(\Gr_{\SO_{N-1}}) @>\Upsilon>>  D^b_{\SO(N-1,\bO)}(\Gr_{\SO_N})
    \end{CD}
  \end{equation}
  commutes, where $\boldsymbol\beta$ stands for the second equivalence of~\cite[Theorem 5]{bf}.
  But by the same~\cite[Theorem 5]{bf}, the deformation
  $D^b_{\SO(N-1,\bO)\rtimes\BG_m}(\Gr_{\SO_{N-1}})$ of $D^b_{\SO(N-1,\bO)}(\Gr_{\SO_{N-1}})$ induces the
  standard Poisson structure on $\fso(V_0)^*$. It follows that
  $\bq_0^*\colon\BC[\fso(V_0)^*]\to\fE$ is Poisson.

  Finally, $\fD^\bullet$ is a unique graded $\BC[\hbar]$-algebra with
  $\fD^\bullet/(\hbar=0)=\Sym(\Pi(V_0\otimes V_1)[-1])$ such that the corresponding Poisson
  bracket on $\Sym(\Pi(V_0\otimes V_1)[-1])$ is the standard one. Thus the desired isomorphism
  $\phi_\hbar$ is constructed along with equivalence $\Phi_\hbar$.
\end{proof}

\subsection{Gaiotto conjectures}
\label{gaiotto}
We recall the setup and notation of~\cite[\S2]{bfgt}. Given a nonnegative integer $m$ such that
$2m+1\leq N$ we set $M=N-1-2m$ and consider an orthogonal decomposition
$\BC^N=\BC^{2m}\oplus\BC^{M+1}$. Furthermore, we choose an anisotropic vector $v\in\BC^{M+1}$,
and set $\BC^M=(\BC v)^\perp$. It gives rise to an embedding
$\SO_M\hookrightarrow\SO_{M+1}\hookrightarrow\SO_N$.
We choose a complete self-orthogonal flag
\[0\subset\sL^1\subset\sL^2\subset\ldots\subset\sL^{2m-1}\subset\BC^{2m},\ \sL^i=(\sL^{2m-i})^\perp.\]
We consider the following partial flag in $\BC^N$:
\[0\subset\sL^1\subset\ldots\subset\sL^m\subset\sL^m\oplus\BC^{M+1}\subset
\sL^{m+1}\oplus\BC^{M+1}\subset\ldots\subset\sL^{2m-1}\oplus\BC^{M+1}\subset\BC^N.\]
We consider a unipotent subgroup $U_{M,N}\subset\SO_N$ with Lie algebra $\fu_{M,N}\subset\fso_N$
formed by all the endomorphisms preserving the above partial flag and inducing the zero
endomorphism of the associated graded space. The composition with orthogonal projection
$\BC^N\twoheadrightarrow\BC^{2m}$ gives rise to a morphism
$U_{M,N}\twoheadrightarrow U_{2m}$ onto the upper triangular unipotent subgroup of $\SO_{2m}$.
Note that this morphism is {\em not} a homomorphism. Nevertheless,
composing this morphism with a regular character $U_{2m}\to\BG_a$ we obtain a character
$\chi'_{M,N}\colon U_{M,N}\to\BG_a$. Furthermore, we choose a vector $\ell\in\sL^m\setminus\sL^{m-1}$.
Then the matrix coefficient $u\mapsto(uv,\ell)$ defines a character $\fu_{M,N}\to\BC$.
The corresponding character $U_{M,N}\to\BG_a$ will be denoted $\chi''_{M,N}$. Finally, we set
$\chi^0_{M,N}:=\chi'_{M,N}+\chi''_{M,N}\colon U_{M,N}\to\BG_a$.
Note that the pair $(U_{M,N},\chi^0_{M,N})$ is invariant under the conjugation action of
$\SO_M\subset\SO_N$.

We extend scalars to the Laurent series field $\bF$ to obtain the same named character of
$U_{M,N}(\bF)$. We define \[\chi_{M,N}:=\Res_{t=0}\chi^0_{M,N}\colon U_{M,N}(\bF)\to\BG_a.\]
Let $\kappa_N$ stand for the bilinear form $\frac12\on{Tr}(X\cdot Y)$ on $\fso_N$.
It corresponds to the determinant line bundle on $\Gr_{\SO_N}$ (the ample generator of the
Picard group). Given $c\in\BC^\times$ we consider the derived category
$D_{c^{-1}}^{\SO(M,\bO)\ltimes U_{M,N}(\bF),\chi_{M,N}}(\Gr_{\SO_N})$ of
$(\SO(M,\bO)\ltimes U_{M,N}(\bF),\chi_{M,N})$-equivariant $D$-modules on $\Gr_{\SO_N}$ twisted
by $c^{-1}\kappa_N$.

\bigskip

On the dual side, in the odd case~\S\ref{OSE}(a), we consider the Lie superalgebra
$\osp(2n-2m|2n)$. In the even case~\S\ref{OSE}(b), we consider the Lie superalgebra
$\osp(2n|2n-2m-2)$. The Killing form $\Killing_{2n-2m|2n}$ (resp.\ $\Killing_{2n|2n-2m-2}$)
is proportional to the supertrace form of the defining representation
$\kappa_{2n-2m|2n}(X,Y)=\on{sTr}(X\cdot Y)$:
\[\Killing_{2n-2m|2n}=(-2m-2)\kappa_{2n-2m|2n}\ (\on{resp.}\ \Killing_{2n|2n-2m-2}=2m\kappa_{2n|2n-2m-2}),\]
see~\cite[2.7.7.(c)]{m}. For $c\in\BC$ we consider the derived Kazhdan-Lusztig category
$\on{KL}_c(\widehat\osp(2n-2m|2n))$ of $\SO(2n-2m,\bO)\times\Sp(2n,\bO)$-equivariant objects
in $\widehat\osp(2n-2m|2n)$-mod at central charge $c\cdot\kappa_{2n-2m|2n}-\frac12\Killing_{2n-2m|2n}$
(resp.\ the derived category $\on{KL}_c(\widehat\osp(2n|2n-2m-2))$ of
$\SO(2n,\bO)\times\Sp(2n-2m-2,\bO)$-equivariant objects
in $\widehat\osp(2n|2n-2m-2)$-mod at central charge
$c\cdot\kappa_{2n|2n-2m-2}-\frac12\Killing_{2n|2n-2m-2}$).

\begin{conj}
  \label{davide}
  \textup{(a)} In the odd case~\ref{OSE}(a), for $c\in\BC^\times$ the categories
  $D_{c^{-1}}^{\SO(M,\bO)\ltimes U_{M,N}(\bF),\chi_{M,N}}(\Gr_{\SO_N})$ and
  $\on{KL}_c(\widehat\osp(2n-2m|2n))$ are equivalent as factorization categories.

  \textup{(b)} In the even case~\ref{OSE}(b), for $c\in\BC^\times$ the categories
  $D_{c^{-1}}^{\SO(M,\bO)\ltimes U_{M,N}(\bF),\chi_{M,N}}(\Gr_{\SO_N})$ and
  $\on{KL}_c(\widehat\osp(2n|2n-2m-2))$ are equivalent as factorization categories.
\end{conj}

\begin{rem}
  \label{d21}
  Let $N=4,\ M=3$. Then $\SO_4\cong(\on{SL}_2\times\on{SL}_2)/\{\pm1\}$ (quotient by
  the diagonal central subgroup), so each connected component of $\Gr_{\SO_4}$ is isomorphic to
  $\Gr_{\on{SL}_2}\times\Gr_{\on{SL}_2}$. Hence the Picard group of each connected component has
  rank $2$, and we have a {\em 2-parametric} family of twistings of $D$-modules on $\Gr_{\SO_4}$.
  On the dual side, we have a family $D(2,1;\alpha)$ of deformations of $\osp(4|2)$.
  It is expected that the categories of twisted $\SO_3$-equivariant $D$-modules on $\Gr_{\SO_4}$
  are equivalent to the corresponding Kazhdan-Lusztig categories for the affine Lie superalgebras
  $D(2,1;\alpha)^{(1)}$.
  \end{rem}

\subsection{Orthosymplectic Kostka polynomials}
\label{kostka}
We will use notation and results of~\cite[Chapter 3]{m}. Recall that Borel subalgebras of
$\osp(V_0|V_1)$ containing $\fb_0\oplus\fb_1\subset\fso(V_0)\oplus\fsp(V_1)$ (notation
of~\S\ref{setup}) are parametrized by {\em shuffles}~\cite[\S3.3]{m} (certain permutations
of the set $\{1,2,\ldots,2n\}$ in the odd case~\ref{OSE}(a) (resp.\ of the set
$\{1,2,\ldots,2n-1\}$ in the even case~\ref{OSE}(b))). We will need a shuffle
\[\sigma^N=(n+1,1,n+2,2,\ldots,2n-1,n-1,2n,n)\] in the odd case and
\[\sigma^N=(1,n+1,2,n+2,\ldots,n-1,2n-1,n)\] in the even case. Note that
$\sigma^N$ is of type $D$~\cite[page 35]{m}.
The corresponding Borel subalgebra of $\osp(V_0|V_1)$ will be denoted $\fb^N$.
This is the so called {\em mixed} Borel subalgebra of~\cite[\S4]{gl}.
Its radical will be denoted by $\fn^N$.
According to~\cite[Lemma 3.3.7(c)]{m}, the odd part $\fn^N_{\bar1}$ has Cartan
eigenvalues
\begin{equation}
  \label{n odd}
  R^{N+}_{\bar1}=\{\varepsilon_i+\delta_j\}_{1\leq i,j\leq n}\cup
\{\varepsilon_i-\delta_j\}_{1\leq i<j\leq n}\cup
\{\delta_i-\varepsilon_j\}_{1\leq i\leq j\leq n}
\end{equation}
in the odd case, and
\begin{equation}
  \label{n even}
R^{N+}_{\bar1}=\{\varepsilon_i+\delta_j\}_{1\leq i\leq n}^{1\leq j<n}\cup
\{\varepsilon_i-\delta_j\}_{1\leq i\leq j<n}\cup
\{\delta_i-\varepsilon_j\}_{1\leq i<j\leq n}
\end{equation}
in the even case.
Thus $\fn^N_{\bar1}$ is a Lagrangian subspace in $V_0\otimes V_1$.
The set of simple roots of $\fn^N_{\bar1}$ is
\[\{\delta_1-\varepsilon_1,\varepsilon_1-\delta_2,\delta_2-\varepsilon_2,\ldots,
\delta_{n-1}-\varepsilon_{n-1},\varepsilon_{n-1}-\delta_n,\delta_n-\varepsilon_n,
\delta_n+\varepsilon_n\},\]
in the odd case, and
\[\{\varepsilon_1-\delta_1,\delta_1-\varepsilon_2,\varepsilon_2-\delta_2,\ldots,
\delta_{n-2}-\varepsilon_{n-1},\varepsilon_{n-1}-\delta_{n-1},\delta_{n-1}-\varepsilon_n,
\delta_{n-1}+\varepsilon_n\},\]
in the even case, cf.~\cite[Lemma 3.4.3(e)]{m}. All the simple roots are odd isotropic.

\bigskip

Given $\alpha\in\ft_0^*\oplus\ft_1^*$ we define a polynomial $L_\alpha^N(q)$
as follows: $L_\alpha^N(q):=\sum p_d^N q^d$ where
$p_d^N$ is the number of (unordered) partitions of $\alpha$ into a sum of $d$
elements of $R^{N+}_{\bar1}$.

\begin{defn}
  \label{kost}
  \textup{(a)} Given
  $\lambda_0,\mu_0\in\Lambda_0^+,\ \lambda_1,\mu_1\in\Lambda_1^+$, we define the
  orthosymplectic Kostka polynomial $K^N_{(\lambda_0,\lambda_1),(\mu_0,\mu_1)}(q)$ by the
  following Lusztig-Kato formula (cf.~\cite[(9.4)]{l},~\cite[Theorem 1.3]{k} and~\cite[(2.1)]{p}):
  \[K^N_{(\lambda_0,\lambda_1),(\mu_0,\mu_1)}(q)=\sum_{w_0\in W_0,\ w_1\in W_1}(-1)^{w_0}(-1)^{w_1}
  L^N_{(w_0(\lambda_0+\rho_0)-\rho_0-\mu_0,w_1(\lambda_1+\rho_1)-\rho_1-\mu_1)}(q),\]
  notation of~\S\ref{setup}.

  \textup{(b)} We say that
  $(\lambda_0,\lambda_1)\geq(\mu_0,\mu_1)$ if
  $(\lambda_0,\lambda_1)-(\mu_0,\mu_1)\in\BN\langle R^{N+}_{\bar1}\rangle$.
\end{defn}

In more concrete terms, recall that $\lambda_0$ is a collection of integers
$(\lambda_0^{(1)},\ldots,\lambda_0^{(n)})$ such that
$\lambda_0^{(1)}\geq\lambda_0^{(2)}\geq\ldots\geq\lambda_0^{(n-1)}\geq|\lambda_0^{(n)}|$, while
$\lambda_1$ is a partition of length $n$
$\lambda_1^{(1)}\geq\lambda_1^{(2)}\geq\ldots\geq\lambda_1^{(n)}$
in the odd case (resp.\ of length $n-1$ in the even case).
In the odd case $(\lambda_0,\lambda_1)\geq(\mu_0,\mu_1)$ if
\begin{multline}
  \label{odd order}
\lambda_1^{(1)}\geq\mu_1^{(1)},\ \lambda_1^{(1)}+\lambda_0^{(1)}\geq\mu_1^{(1)}+\mu_0^{(1)},\ldots,\\
\lambda_1^{(1)}+\lambda_0^{(1)}+\ldots+\lambda_0^{(n-1)}+\lambda_1^{(n)}\geq
\mu_1^{(1)}+\mu_0^{(1)}+\ldots+\mu_0^{(n-1)}+\mu_1^{(n)},\\
\lambda_1^{(1)}+\lambda_0^{(1)}+\ldots+\lambda_1^{(n)}+\lambda_0^{(n)}\in2\BN+
\mu_1^{(1)}+\mu_0^{(1)}+\ldots+\mu_1^{(n)}+\mu_0^{(n)},\\
\lambda_1^{(1)}+\lambda_0^{(1)}+\ldots+\lambda_1^{(n)}-\lambda_0^{(n)}\geq
\mu_1^{(1)}+\mu_0^{(1)}+\ldots+\mu_1^{(n)}-\mu_0^{(n)}.
\end{multline}
In the even case $(\lambda_0,\lambda_1)\geq(\mu_0,\mu_1)$ if
\begin{multline}
  \label{even order}
  \lambda_0^{(1)}\geq\mu_0^{(1)},\ \lambda_0^{(1)}+\lambda_1^{(1)}\geq\mu_0^{(1)}+\mu_1^{(1)},\ldots,\\
\lambda_0^{(1)}+\lambda_1^{(1)}+\ldots+\lambda_0^{(n-1)}+\lambda_1^{(n-1)}\geq
\mu_0^{(1)}+\mu_1^{(1)}+\ldots+\mu_0^{(n-1)}+\mu_1^{(n-1)},\\
\lambda_0^{(1)}+\lambda_1^{(1)}+\ldots+\lambda_1^{(n-1)}+\lambda_0^{(n)}\in2\BN+
\mu_0^{(1)}+\mu_1^{(1)}+\ldots+\mu_1^{(n-1)}+\mu_0^{(n)},\\
\lambda_0^{(1)}+\lambda_1^{(1)}+\ldots+\lambda_1^{(n-1)}-\lambda_0^{(n)}\geq
\mu_0^{(1)}+\mu_1^{(1)}+\ldots+\mu_1^{(n)}-\mu_0^{(n)}.
\end{multline}
(In both cases, the first three lines compare partial sums of the shuffled sequences
$(\la_1^{(1)},\la_0^{(1)},\dots,\la_1^{(n)},\la_0^{(n)})$ and
$(\mu_1^{(1)},\mu_0^{(1)},\dots,\mu_1^{(n)},\mu_0^{(n)})$ in the odd case, resp.\
$(\la_0^{(1)},\la_1^{(1)},\dots,\la_0^{(n-1)},\la_1^{(n-1)},\la_0^{(n)})$ and
$(\mu_0^{(1)},\mu_1^{(1)},\dots,\mu_0^{(n-1)},\mu_1^{(n-1)},\mu_0^{(n)})$ in the even case.)

\bigskip

Recall that $\fn^N_{\bar1}$ is a $B_0\times B_1$-module for the
adjoint action (here $B_0\subset\SO(V_0)$ and $B_1\subset\Sp(V_1)$ are the Borel subgroups
with Lie algebras $\fb_0\subset\fso(V_0),\ \fb_1\subset\fsp(V_1)$ respectively,
see~\S\ref{setup}). We denote by $\widetilde\CN{}^N_{\bar1}$ the associated
vector bundle over the flag variety $\CB_0\times\CB_1:=\SO(V_0)/B_0\times\Sp(V_1)/B_1$.

To a pair $(\mu_0,\mu_1)\in\Lambda_0^+\oplus\Lambda_1^+$ we associate the
$\SO(V_0)\times\Sp(V_1)$-equivariant line bundle $\CO(\mu_0,\mu_1)$ on the flag
variety $\CB_0\times\CB_1$: the action of $B_0\times B_1$ on its fiber over the point
$(B_0,B_1)\in\CB_0\times\CB_1$ is via the character $(-\mu_0,-\mu_1)$. Its global sections
$\Gamma(\CB_0\times\CB_1,\CO(\mu_0,\mu_1))$ is the irreducible $\SO(V_0)\times\Sp(V_1)$-module
$V_{\mu_0^*}\otimes V_{\mu_1^*}$ with lowest weight $(-\mu_0,-\mu_1)$. The character of
$V_{\mu_0^*}\otimes V_{\mu_1^*}$ will be denoted by $\chi_{(\mu_0^*,\mu_1^*)}$.

The pullback of $\CO(\mu_0,\mu_1)$ to $\widetilde\CN{}^N_{\bar1}$ will be also
denoted $\CO(\mu_0,\mu_1)$. We consider the graded equivariant Euler characteristics
\[\chi(\widetilde\CN{}^N_{\bar1},\CO(\mu_0,\mu_1))=
\chi(\CB_0\times\CB_1,\Sym{}\!^\bullet\ul\fn{}_{\bar1}^{N\vee}\otimes\CO(\mu_0,\mu_1)):\]
formal Taylor power series in $q$ with coefficients in the character ring of
$\SO(V_0)\times\Sp(V_1)$. Here $\ul\fn{}_{\bar1}^N$ is the sheaf of
sections of the $\SO(V_0)\times\Sp(V_1)$-equivariant vector bundle over $\CB_0\times\CB_1$
associated to the $B_0\times B_1$-module $\fn_{\bar1}^N$.
In other words, $\ul\fn{}_{\bar1}^N$ is the sheaf of sections of $\widetilde\CN{}^N_{\bar1}$
viewed as a vector bundle over $\CB_0\times\CB_1$.

\begin{prop}[D.~Panyushev]
  \label{bryl}
  We have
  \[\chi(\CB_0\times\CB_1,\Sym{}\!^\bullet\ul\fn{}_{\bar1}^{N\vee}\otimes\CO(\mu_0,\mu_1))=
  \sum_{(\lambda_0,\lambda_1)\geq(\mu_0,\mu_1)}K^N_{(\lambda_0,\lambda_1),(\mu_0,\mu_1)}(q)
  \chi_{(\lambda_0^*,\lambda_1^*)}.\]
\end{prop}

\begin{proof}
  This is a particular case of~\cite[Theorem 3.8]{p}.
\end{proof}

\begin{cor}
  \label{pan}
  For any $(\lambda_0,\lambda_1)\geq(\mu_0,\mu_1)$ we have
  \[K^N_{(\lambda_0,\lambda_1),(\mu_0,\mu_1)}(q)\in\BN[q].\]
\end{cor}

\begin{proof}
  The desired positivity follows from the higher cohomology vanishing
  $R^{>0}\Gamma(\widetilde\CN{}^N_{\bar1},\CO(\mu_0,\mu_1))=0$.
  Note that the canonical class of $\widetilde\CN{}^N_{\bar1}$
  is $\SO(V_0)\times\Sp(V_1)$-equivariantly trivial. Indeed, a straightforward calculation
  shows that the sum of all elements of $R^{N+}_{\bar1}$
  equals $2\rho_0+2\rho_1$. But the set of Cartan eigenvalues in the
  fiber of the tangent bundle $T_{(B_0,B_1)}\CB_0\times\CB_1$ coincides with the set of
  {\em negative} roots, and they sum up to $-2\rho_0-2\rho_1$. Note that in the language of Lie
  superalgebras, the canonical class vanishing is equivalent to the equality $2\rho=0$, where
  $2\rho$ is the sum of all even roots in a mixed Borel subgroup minus the sum of all odd roots
  in this Borel subgroup. The equality $2\rho=0$ follows from the fact that all
  the simple roots of a mixed Borel subgroup are odd isotropic~\cite[Corollary 8.5.4]{m}.

  We have a proper projection \[\widetilde\CN{}^N_{\bar1}\to V_0\otimes V_1=\Pi\osp(V_0|V_1)_{\bar1}\]
  birational onto its image
  (odd nilpotent cone $\CN_{\bar1}^N$, see~\cite[Th\'eor\`eme 1]{gl}
  and~\cite[Theorem 2.3.5]{mo}\footnote{In fact, this resolution
  of the orthosymplectic odd nilpotent cone is a particular case of a general construction~\cite{h}.
  We are grateful to A.~Elashvili for this observation.}). Now the desired cohomology
  vanishing follows by the Kempf collapsing as in the proof of~\cite[Theorem 3.1.(ii)]{p}.
  \end{proof}

\begin{rem}
  In~\cite[D\'efinition 5.1)]{gl} Gruson and Leidwanger define a mixed Borel subalgebra
  in $\osp(2n+1|2n)$ (in fact, they define mixed Borel subalgebras in arbitrary orthosymplectic
  Lie superalgebras). An obvious modification of~Definition~\ref{kost} produces Kostka
  polynomials in this case (and for mixed Borel subalgebras in arbitrary orthosymplectic
  Lie superalgebras). However, the proof of positivity Corollary~\ref{pan} fails since
  $\rho\ne0$ (not all the simple roots are isotropic), cf.~\cite[Proposition 4.0.1]{mo}.
  It would be interesting to know if the positivity still holds true in this case.
\end{rem}

\begin{thm}
  \label{ic stalk}
  \textup{(a)} In the odd case~\ref{OSE}(a), an $\SO(N-1,\bO)$-orbit
  $\BO^{\mu_0}_{\mu_1}\subset\Gr_{\SO_N}$ lies in the closure of
  $\BO^{\lambda_0}_{\lambda_1}$ iff $(\lambda_0,\lambda_1)\geq(\mu_0,\mu_1)$.

  \textup{(b)} In the even case~\ref{OSE}(b), an $\SO(N-1,\bO)$-orbit
  $\BO^{\mu_1}_{\mu_0}\subset\Gr_{\SO_N}$ lies in the closure of
  $\BO^{\lambda_1}_{\lambda_0}$ iff $(\lambda_0,\lambda_1)\geq(\mu_0,\mu_1)$.

  \textup{(c)} In the odd case we have
  \[q^{-\dim\BO^{\mu_0}_{\mu_1}}K^N_{(\lambda_0,\lambda_1),(\mu_0,\mu_1)}(q^{-1})=
  \sum_i\dim(\IC^{\lambda_0}_{\lambda_1})_{\BO^{\mu_0}_{\mu_1}}^{-i}q^{-i}.\]

  \textup{(d)} In the even case we have
  \[q^{-\dim\BO^{\mu_1}_{\mu_0}}K^N_{(\lambda_0,\lambda_1),(\mu_0,\mu_1)}(q^{-1})=
  \sum_i\dim(\IC^{\lambda_1}_{\lambda_0})_{\BO^{\mu_1}_{\mu_0}}^{-i}q^{-i},\]
  (the Poincar\'e polynomials of the $\IC^{\lambda_\sfs}_{\lambda_{\sfb}}$-stalks at the orbit
  $\BO^{\mu_\sfs}_{\mu_\sfb}$).
\end{thm}

\begin{proof}
(a) We first prove that if $\BO^{\mu_0}_{\mu_1}\subset\ol\BO{}^{\lambda_0}_{\lambda_1}$ then
  $(\lambda_0,\lambda_1)\geq(\mu_0,\mu_1)$. We view $\BO^{\mu_0}_{\mu_1},\BO^{\lambda_0}_{\lambda_1}$
  as connected components of the fixed point sets of involution $\varsigma$ of the
  corresponding $\GL(N-1,\bO)$-orbits in $\Gr_{\GL_N}$ as in the proofs
  of~Lemmas~\ref{sake?},\ref{connected}. Recall that the set of $\GL(N-1,\bO)$-orbits in
  $\Gr_{\GL_N}$ is parametrized by bisignatures \[(\btheta_0,\btheta_1)=
  (\theta_0^{(1)}\geq\ldots\geq\theta_0^{(N-1)},\theta_1^{(1)}\geq\ldots\geq\theta_1^{(N)}),\]
  and the adjacency order on orbits is given on bisignatures by
  $(\btheta_0,\btheta_1)\geq(\bzeta_0,\bzeta_1)$ if
  \[\theta_1^{(1)}\geq\zeta_1^{(1)},\ \theta_1^{(1)}+\theta_0^{(1)}\geq\zeta_1^{(1)}+\zeta_0^{(1)},\
  \theta_1^{(1)}+\theta_0^{(1)}+\theta_1^{(2)}\geq\zeta_1^{(1)}+\zeta_0^{(1)}+\zeta_1^{(2)},\ldots,\]
  \[\theta_1^{(1)}+\theta_0^{(1)}+\ldots+\theta_1^{(N-1)}+\theta_0^{(N-1)}\geq
  \zeta_1^{(1)}+\zeta_0^{(1)}+\ldots+\zeta_1^{(N-1)}+\zeta_0^{(N-1)},\]
  \[\theta_1^{(1)}+\theta_0^{(1)}+\ldots+\theta_0^{(N-1)}+\theta_1^{(N)}=
  \zeta_1^{(1)}+\zeta_0^{(1)}+\ldots+\zeta_0^{(N-1)}+\zeta_1^{(N)}.\]
  Indeed, the similar description of the adjacency order on the set of $\GL(N,\bO)$-orbits in the
  mirabolic Grassmannian is given in~\cite[Proposition 12]{fgt} as a corollary
  of~\cite[Theorem~3.9]{ah}. The desired description of the adjacency order on the set of
  $\GL(N-1,\bO)$-orbits in $\Gr_{\GL_N}$ follows by the arguments of~\cite[\S4.4]{bfgt}.

  Now if $\BO^{\mu_0}_{\mu_1}\subset\ol\BO{}^{\lambda_0}_{\lambda_1}$, then the $\GL(N-1,\bO)$-orbit
  in $\Gr_{\GL_N}$ containing $\BO^{\mu_0}_{\mu_1}$ (note that it depends only on the bipartition
  $(|\mu_0|:=(\mu_0^{(1)}\geq\ldots\geq\mu_0^{(n-1)}\geq|\mu_0^{(n)}|),\ \mu_1)$)
  lies in the closure of the $\GL(N-1,\bO)$-orbit
  in $\Gr_{\GL_N}$ containing $\BO^{\lambda_0}_{\lambda_1}$.
  This implies the first three lines of inequalities~(\ref{odd order}) for $(|\la_0|,\la_1)$ and
  $(|\mu_0|,\mu_1)$. The following trick takes
  care of the last inequality of~(\ref{odd order}). Take any dominant coweight
  $\nu_0=(\nu_0^{(1)},\ldots,\nu_0^{(n)})$ (such that
  $\nu_0^{(1)}\geq\nu_0^{(2)}\geq\ldots\geq\nu_0^{(n-1)}\geq|\nu_0^{(n)}|$) of $\SO_{N-1}$.
  Consider the corresponding convolution
  $\bm(\ol\Gr{}_{\SO_{N-1}}^{\nu_0\mathstrut}\wtimes\ol\BO{}^{\lambda_0}_{\lambda_1})
  =\ol\BO{}^{\nu_0+\lambda_0}_{\lambda_1}$. Since
  $\ol\BO{}^{\mu_0}_{\mu_1}\subset\ol\BO{}^{\lambda_0}_{\lambda_1}$, applying convolution with
  $\ol\Gr{}_{\SO_{N-1}}^{\nu_0}$ to both sides, we deduce
  $\ol\BO{}^{\nu_0+\mu_0}_{\mu_1}\subset\ol\BO{}^{\nu_0+\lambda_0}_{\lambda_1}$. But the first three lines
  of~(\ref{odd order}) for $(|\nu_0+\lambda_0|,\lambda_1),(|\nu_0+\mu_0|,\mu_1)$ with arbitrary
  $\nu_0$ imply the last two (and thus all) lines of~(\ref{odd order}) for
  $(\lambda_0,\lambda_1),(\mu_0,\mu_1)$.

  This completes the proof of the `only if' direction of (a). The proof of the `only if' direction
  of (b) is entirely similar. We will return to the proof of the `if' directions of (a,b) after the
  proof of (c,d).

  \medskip

  (c) We choose a base point in an orbit $\BO^{\mu_0}_{\mu_1}$ (e.g.\ the one supplied
  in the proofs of~Lemmas~\ref{sake?},\ref{connected}) and denote by
  $\bp\colon\SO(N-1,\bO)\to\BO^{\mu_0}_{\mu_1}$ the corresponding action morphism. We denote by
  $C^{\mu_0}_{\mu_1}$ the direct image of the constant sheaf $\bp_*\ul\BC{}_{\SO(N-1,\bO)}$.
  Note that the action of $\SO(N-1,\bO)$ on $\BO^{\mu_0}_{\mu_1}$ factors through the quotient
  $\SO(N-1,\bO)/U$ by a normal unipotent subgroup of finite codimension, so that $\bp$
  factors through $\bp'\colon\SO(N-1,\bO)/U\to\BO^{\mu_0}_{\mu_1}$, and the rigorous definition
  of $C^{\mu_0}_{\mu_1}$ is $\bp'_*\ul\BC{}_{\SO(N-1,\bO)/U}$. It is canonically independent of the
  choice of $U$, hence our notation $\bp_*\ul\BC{}_{\SO(N-1,\bO)}$.

  Now $\Hom^\bullet_{D^b_{\SO(N-1,\bO)}(\Gr_{\SO_N})}(\IC^{\lambda_0}_{\lambda_1},C^{\mu_0}_{\mu_1})$ is canonically
  dual to the stalk of $\IC^{\lambda_0}_{\lambda_1}$ at the orbit $\BO^{\mu_0}_{\mu_1}$. So it suffices
  to prove that under the equivalence $\Phi$ of~Theorem~\ref{main} we have
  $\Phi(\CalC^{\mu_0}_{\mu_1})\simeq C^{\mu_0}_{\mu_1}$, where
  $\CalC^{\mu_0}_{\mu_1}\in D^{\sG_{\bar0}}_\perf(\fG^\bullet)$ is the following dg-module.
It is equal to the global sections $\Gamma(\widetilde\CN{}^N_{\bar1},\CO(\mu_0,\mu_1))$
(cf.\ the proof of~Corollary~\ref{pan}) equipped with the trivial differential and the grading
coming from the dilation action of $\BC^\times$ on $\widetilde\CN{}^N_{\bar1}$ and the natural
$\BC^\times$-equivariant structure of the line bundle $\CO(\mu_0,\mu_1)$ on
$\widetilde\CN{}^N_{\bar1}$. The $\fG^\bullet$-module structure comes from the natural
$\BC[V_0\otimes V_1]$-module structure on $\Gamma(\widetilde\CN{}^N_{\bar1},\CO(\mu_0,\mu_1))$
and the above grading.

The isomorphism class of $C^{\mu_0}_{\mu_1}$ is uniquely characterized by the following properties:

(i) If $\Hom^\bullet_{D^b_{\SO(N-1,\bO)}(\Gr_{\SO_N})}(\IC^{\lambda_0}_{\lambda_1},C^{\mu_0}_{\mu_1})\ne0$, then
$\BO^{\mu_0}_{\mu_1}\subset\ol\BO{}^{\lambda_0}_{\lambda_1}$, and
$\Hom^\bullet_{D^b_{\SO(N-1,\bO)}(\Gr_{\SO_N})}(\IC^{\mu_0}_{\mu_1},C^{\mu_0}_{\mu_1})=\BC[-\dim\BO^{\mu_0}_{\mu_1}]$;

(ii) $C^{\mu_0}_{\mu_1}$ lies in the triangulated subcategory of $D^b_{\SO(N-1,\bO)}(\Gr_{\SO_N})$
generated by $\{\IC^{\nu_0}_{\nu_1}\}$ for pairs $(\nu_0,\nu_1)$ such that
$\BO^{\nu_0}_{\nu_1}\subset\ol\BO{}^{\mu_0}_{\mu_1}$.

So we have to check the corresponding properties of $\CalC^{\mu_0}_{\mu_1}$. Due to the `only if'
direction of part (a) proved above, we may replace the closure relations by the inequalities
$(\lambda_0,\lambda_1)\geq(\mu_0,\mu_1)$ in (i) (resp.\ $(\nu_0,\nu_1)\leq(\mu_0,\mu_1)$ in (ii)).
To check (ii) we consider the $\sG_{\bar0}$-module $(\CalC^{\mu_0}_{\mu_1})_0$ equal to
$\on{Tor}^{\fG^\bullet}(\CalC^{\mu_0}_{\mu_1},\BC)$, where $\BC$ is the quotient of $\fG^\bullet$
modulo the augmentation ideal. We have to
verify that if an irreducible $\sG_{\bar0}$-module $V_{\nu_0}\otimes V_{\nu_1}$ enters
$(\CalC^{\mu_0}_{\mu_1})_0$ with nonzero
multiplicity, then $(\nu_0,\nu_1)\leq(\mu_0,\mu_1)$. We apply the base
change~\cite[Proposition 2.5.14]{dag} for the Cartesian square
\[\begin{CD}
\widetilde{\CB_0\times\CB_1} @>>> \widetilde\CN{}^N_{\bar1}\\
@VVV @VVV\\
0 @>>> V_0\otimes V_1.
\end{CD}\]
Here $\widetilde{\CB_0\times\CB_1}$ is a derived scheme supported at the zero section of
$\widetilde\CN{}^N_{\bar1}$ with the structure sheaf
$\Lambda^\bullet\left((\CV_0\otimes\CV_1)/\ul\fn{}_{\bar1}^N[-1]\right)^\vee$,
where $\CV_0\otimes\CV_1$ is the trivial vector bundle on $\CB_0\times\CB_1$ with
fiber $V_0\otimes V_1$.
It follows that the $\sG_{\bar0}$-module $(\CalC^{\mu_0}_{\mu_1})_0$ equals
$R\Gamma(\CB_0\times\CB_1,\Lambda^\bullet(\ul\fn{}_{\bar1}^N[1])\otimes\CO(\mu_0,\mu_1))$
(here the exterior algebra of the shifted vector bundle denotes a finite dimensional
algebra as opposed to the symmetric one).
Indeed, the vector bundle $\ul\fn{}_{\bar1}^N$ over $\CB_0\times\CB_1$ is by
construction embedded into the trivial vector bundle $\CV_0\otimes\CV_1$ as a
Lagrangian subbundle, so the quotient
$(\CV_0\otimes\CV_1)/\ul\fn{}_{\bar1}^N$ is canonically
identified with the dual vector bundle $\ul\fn{}_{\bar1}^{N\vee}$.

Now the verification of (i,ii) is the same as the one for steps (i,ii) of the proof
of~\cite[Theorem 5.4]{p}. This completes the proof of (c). The proof of (d) is entirely similar.
Finally, we return to the proof of the `if' direction of (a,b). The arguments in the odd and even
cases being similar, we consider the odd case only.

Since the stalks of $\IC^{\lambda_0}_{\lambda_1}$ do not vanish precisely at the orbits
$\BO^{\mu_0}_{\mu_1}$ lying in the closure of $\BO^{\lambda_0}_{\lambda_1}$, and the stalks are known
by (c), it remains to check that $K^N_{(\lambda_0,\lambda_1),(\mu_0,\mu_1)}\ne0$ if
$(\lambda_0,\lambda_1)\geq(\mu_0,\mu_1)$. From~Definition~\ref{kost} it is easy to see that
the summand of $L^N_{(\lambda_0-\mu_0,\lambda_1-\mu_1)}$ of highest degree (corresponding to the
decomposition of $(\lambda_0,\lambda_1)-(\mu_0,\mu_1)$ into the sum of simple roots) cannot be
cancelled by any other summands in the definition of $K^N_{(\lambda_0,\lambda_1),(\mu_0,\mu_1)}$.
Thus we conclude that $K^N_{(\lambda_0,\lambda_1),(\mu_0,\mu_1)}\ne0$.

The theorem is proved.
\end{proof}

\end{document}